\documentclass{amsart}[12pt]
\usepackage{amsmath, amsthm, amscd, amsfonts,}
\usepackage{graphicx,float}

\usepackage{amssymb}

\usepackage{pb-diagram}
\usepackage{lamsarrow}
\usepackage{pb-lams}
\usepackage{mathrsfs}
\usepackage{eufrak}
\DeclareMathOperator{\lcf}{lcf}

\def\MPB{{\mathbb{P}}}
\def\MQB{{\mathbb{Q}}}

\def\k{\kappa}

\def\l{\lambda}

\def\a{\alpha}

\def\b{\beta}

\newtheorem{theorem}{Theorem}[section]

\newtheorem{lemma}[theorem]{Lemma}
\newtheorem{proof of Theorem 2.22}[theorem]{Proof of Theorem 2.22}

\newtheorem{corollary}[theorem]{Corollary}
\newtheorem{definition}[theorem]{Definition}
\newtheorem{example}[theorem]{Example}

\newtheorem{remark}[theorem]{Remark}
\newtheorem{claim}[theorem]{Claim}
\numberwithin{equation}{section}

\usepackage{pb-diagram}

\def\l{\lambda}

\def\rmark{\mbox{$\rm\bf\rule{0.06em}{1.45ex}\kern-0.05em R$}}
\def\pmark{\mbox{$\rm\bf\rule{0.06em}{1.45ex}\kern-0.05em P$}}
\def\nmark{\mbox{$\rm\bf\rule{0.06em}{1.45ex}\kern-0.05em N$}}
\def\vdash{\mbox{$\rm\| \kern-0.13em -$}}
\newcommand{\lusim}[1]{\smash{\underset{\raisebox{1.2pt}[0cm][0cm]{$\sim$}}
{{#1}}}}

\begin{document}

\title[On cuts in ultraproducts of linear orders I]{On cuts in ultraproducts of  linear orders I}

\author[ M. Golshani and S. Shelah]{ Mohammad Golshani and Saharon Shelah}

\thanks{The first author's research was in part supported by a grant from IPM (No. 91030417). The second author's research has been partially supported by the European Research Council
grant 338821. This is publication 1075 of second author} \maketitle



\begin{abstract}
For an ultrafilter $D$ on a cardinal $\kappa,$ we wonder for which pair $(\theta_1, \theta_2)$ of regular cardinals,  we have: for any $(\theta_1+\theta_2)^+-$saturated dense linear order $J, J^{\kappa}/ D$ has a cut of cofinality $(\theta_1, \theta_2).$ We deal mainly with the case $\theta_1, \theta_2 > 2^\k.$
\end{abstract}

\section{introduction}
Let $D$ be an ultrafilter  on a cardinal $\kappa.$ We consider the class $\mathscr{C}(D)$ consisting of pairs $(\theta_1, \theta_2),$ where  $(\theta_1, \theta_2)$ is the cofinality of  a cut in $J^\k/D$ and $J$ is some (equivalently any) $(\theta_1+\theta_2)^+-$saturated dense linear order.  The works \cite{malliaris-shelah1}, \cite{malliaris-shelah2} and \cite{malliaris-shelah3} of Malliaris and Shelah have started  the study of this class for the case $\theta_1 + \theta_2 \leq 2^\k$.
In this paper we continue these works; in particular we will concentrate on the case where $\theta_1 + \theta_2$ is above $2^\k$.
As the results of the paper show, the study of the class $\mathscr{C}_{> 2^{\kappa}}(D)$
is very different from the case $\mathscr{C}_{\leq 2^{\kappa}}(D),$ and it is related to the universe of set theory we discuss in it. So most of our results are proved under some set theoretic assumptions, like the existence of large cardinals or the validity of some combinatorial principles, or
are considered in suitable generic extensions of the universe.
We also prove some results about the depth and depth$^+$ of Boolean algebras, which continue the works of Garti-Shelah \cite{garti-shelah1}, \cite{garti-shelah2}, \cite{garti-shelah3}, \cite{garti-shelah4} and Shelah \cite{shelah4}.

In particular, we  prove the following:
\begin{itemize}
\item  Suppose $\k_1>\aleph_0$ and $D$ is a  $\k_1-$complete (but not $\k_1^+-$complete) ultrafilter on $\k_2$ and $(\theta_1, \theta_2)\in \mathscr{C}(D).$ Then $\theta_1, \theta_2 >\k_1$ (Theorem 2.3).
\item Suppose that $D$ is an ultrafilter on $\kappa$, $\kappa<\mu \leq \lambda, \theta,$ where $\mu$ is a strongly compact cardinal and $\lambda, \theta$ are regular. Then $(\lambda, \theta)\notin \mathscr{C}(D)$ (Theorem 2.11).
\item Suppose that $D$ is an ultrafilter on $\kappa$ and $\theta, \lambda$ are regular cardinals such that $\theta^\kappa <\lambda.$ Then $(\lambda, \theta) \notin \mathscr{C}(D)$ (Theorem 2.16).
\item  Assume $V=L$, and  suppose that $D$ is a uniform ultrafilter on some cardinal $\k.$ Then $\mathscr{C}(D)$ is a proper class (Corollary 3.5).
\item If in $V$, there is a class of supercompact cardinals, then for some class forcing $\mathbb{P},$ in $V^\mathbb{P}$ we have: for any infinite cardinal $\k$, and  any ultrafilter $D$ on $\k,$ if $(\l_1, \l_2)\in \mathscr{C}(D),$ then $\l_1 +\l_2 < 2^{2^\k}$ (Theorem 5.21).
\end{itemize}
It follows from our results that the study of the class $\mathscr{C}(D)$ is closely related to large cardinals, and  combinatorial principles like square and diamond; so that in the presence of large cardinals, the class $\mathscr{C}(D)$ is small, while it is a proper class in all known core models like $L$.

We now give some of the main definitions that appear  in the paper.
\begin{definition}
A linear order $J$ is $\tau^+$-saturated, if for every subset $A$ of $J$ of size $\leq \tau$ and every type $\Gamma$ in the language of $J$
with parameters from $A$, if $\Gamma$ is finitely satisfiable, then $\Gamma$ is realized by some element of $J$.
\end{definition}
If $J$ is $\tau^+$-saturated and if $A$ and $B$ are subsets of $J$ of size $\leq \tau$ such that $a <_J b$ for all $a \in A$ and $b \in B,$
then it is easily seen that there are $s_1, s_2$ and $s_3 \in J$ such that for all $a \in A$ and $b \in B,~ s_1 <_J a <_J s_2 <_J b <_J s_3.$ To see this, consider the types $\Gamma_1 = \{ x <_J a : a \in A     \}$, $\Gamma_2 = \{a <_J x <_J b: a \in A$ and $b \in B    \}$ and $\Gamma_3 = \{b <_J x: b \in B    \}$.
They are easily seen to be finitely satisfiable, and hence by the saturation of $J$, they are realized by some $s_1, s_2$ and $s_3$ respectively. Then
$s_1, s_2$ and $s_3$ are as required.
\begin{definition}
Let $J$ be a linear order, and $C_1, C_2 \subseteq J.$

$(a)$ $(C_1, C_2)$ is a pre-cut of $J$ if  $C_1 <_J C_2$ (i.e. for all $s_1\in C_1$ and $s_2\in C_2, s_1 <_J s_2$),

$\hspace{.5cm}$ and there is no $t\in J$ such that $C_1 <_J t <_J C_2.$

$(b)$ $(C_1, C_2)$ is a cut of $J$, if it is a pre-cut of $J$ and $J=C_1 \cup C_2$.

$(c)$ For a pre-cut $(C_1, C_2)$ of $J$, let $cf(C_1, C_2)=(\theta_1, \theta_2)$ where
\begin{itemize}
\item $\theta_1$ is the cofinality of $C_1, \theta_1=cf(C_1),$
\item $\theta_2$ is the initiality (or downward cofinality) of $C_2, \theta_2= dcf(C_2)$.
\end{itemize}

$(d)$ Suppose $(C_1, C_2)$ is a pre-cut of $J$ and  $cf(C_1, C_2)=(\theta_1, \theta_2).$
\begin{itemize}
\item $\bar{s}$ witnesses $cf(C_1)=\theta_1,$ if $\bar{s}=\langle s_\alpha: \alpha<\theta_1  \rangle$ is $<_J-$increasing and unbounded in
 $C_1$.
\item $\bar{t}$ witnesses $dcf(C_2)=\theta_2,$ if $\bar{t}=\langle t_\b: \b<\theta_2  \rangle$ is $<_J-$decreasing and for any
 $t\in C_2$ there exists $\b<\theta_2$ such that $t_\b <_J t.$

\item $(\bar{s}, \bar{t})$ witnesses $cf(C_1, C_2)=(\theta_1, \theta_2),$ if $\bar{s}$ witnesses $cf(C_1)=\theta_1$ and $\bar{t}$ witnesses
$dcf(C_2)=\theta_2.$
\end{itemize}

\end{definition}
\begin{definition}
Suppose $D$ is an ultrafilter on a cardinal $\kappa.$ Then:

$(a)$ $\mathscr{C}(D)=\{ (\theta_1, \theta_2): (\theta_1, \theta_2)=cf(C_1, C_2)$ for some cut $(C_1, C_2)$ of $J^\kappa / D$, where $J$ is

$\hspace{.5cm}$ some (equivalently any) $(\theta_1+\theta_2)^+-$saturated dense linear order $ \}.$

 $(b)$ $\mathscr{C}_{\geq \lambda}(D)=\{ (\theta_1, \theta_2) \in \mathscr{C}(D): \theta_1+\theta_2 \geq \lambda \}.$

$(c)$  $\mathscr{C}_{\leq \lambda}(D)=\{ (\theta_1, \theta_2) \in \mathscr{C}(D): \theta_1+\theta_2 \leq \lambda \}.$

\end{definition}
\begin{remark}
$(a)$ $\mathscr{C}(D)$ is symmetric, i.e. $ (\theta_1, \theta_2)\in \mathscr{C}(D)  \Leftrightarrow    (\theta_2, \theta_1)\in \mathscr{C}(D),$ for all regular cardinals $\theta_1, \theta_2.$

$(b)$ There are $(\theta_1, \theta_2)-$cuts with $\theta_1\in \{0,1\},$ but they do not arise in our work, because if a $\l^+-$saturated dense linear order has a $(\theta_1, \theta_2)-$cut and $\theta_1\in \{0,1\},$ then $\theta_2 > \l.$

$(c)$ By ultrafilter, we always mean a non-principal ultrafilter.
\end{remark}
We also use the following combinatorial principles that are valid in known core models, and  will use them to show that the class $\mathscr{C}(D)$
can be large.
\begin{definition}
Assume $\kappa$ is a regular uncountable cardinal and $S \subseteq \kappa$ is stationary. The diamond principle $\Diamond_S$ is the assertion ``there exists
a sequence $\langle  s_\alpha: \alpha \in S    \rangle$ such that each $s_\alpha$ is a subset of $\alpha$
and for any $X \subseteq \kappa,$ the set $\{\alpha \in S: X \cap \alpha = s_\alpha          \}$
is stationary in $\kappa$''.

\end{definition}
The following is a version of square principle that will be used through this paper.
\begin{definition}
$(a)$ A set $S\subseteq \kappa$ has a square, if $\kappa$ is a regular uncountable cardinal and there exists a set $S^+\subseteq\kappa$ and a sequence $\bar{C}=\langle C_\alpha: \alpha\in S^+ \rangle$ such that:

$\hspace{1.5cm}$$(\alpha)$ $S\setminus S^+$ in a non-stationary subset of $\kappa,$

$\hspace{1.5cm}$$(\beta)$ $ C_\alpha$ is a club of $\alpha,$

$\hspace{1.5cm}$$(\gamma)$ $\beta\in C_\alpha \Rightarrow \beta\in S^+$ and $C_\beta=C_\alpha\cap\beta,$

$\hspace{1.5cm}$$(\delta)$ $otp(C_\a) <\a.$

We may assume $C_\alpha=\emptyset,$ for $\alpha\notin S^+.$

$(b)$ Given a club subset $C$ of a limit ordinal $\alpha,$ let $nacc(C)$ be the set of non-–accumulation points of $C$, i.e.,
$nacc(C)=\{\beta\in C: sup(C\cap\beta)<\beta  \}$.
\end{definition}

For a forcing notion $\MPB,~ p \leq q$ means that $p$ gives more information  than $q$, or $p$ is stronger than $q$. The forcing notion used in this paper is
 the Cohen forcing described below.
\begin{definition}
Assume $\theta$ is a regular cardinal and $I$ is a  set with $|I| \geq \theta$. The Cohen forcing for adding $|I|$-many Cohen subsets of
$\theta,$ indexed by $I$, denoted $Add(\theta, I),$ consists of partial functions $p: I \to \{ 0, 1\}$ of size less than $\theta$,
ordered by reverse inclusion.
\end{definition}
The forcing notion $Add(\theta, I)$ is $\theta$-closed and satisfies  $(2^{<\theta})^+$-c.c., hence if $2^{<\theta}=\theta,$
then it preserves all cardinals.

\section{On $\mathscr{C}(D)$ being small}
In this section we consider the cases where $\mathscr{C}(D)$ is small,  by showing  that $\mathscr{C}(D)$ may not contain some pairs $(\theta, \sigma),$ for some suitable regular cardinals $\theta, \sigma.$ In particular we show that if $D$ is an ultrafilter on $\kappa$ and $\mu>\k$ is strongly compact, then $\mathscr{C}(D)$ is a set (in fact $\mathscr{C}(D) \subseteq \mu\times \mu$).

The following lemma will be useful in some of our arguments. It says instead of a saturated dense linear order, we can work with its completion.
\begin{lemma}
Assume $J_*$ is a $\lambda^+-$saturated dense linear order, $J$ is its completion and $D$ is an ultrafilter on $\kappa.$ Then:

$(a)$ If a cut of $J_*^\kappa/D$ has both cofinalities $\leq \lambda,$ then this cut is not filled in $J^\kappa/D.$

$(b)$ If $(C_1, C_2)$ is a cut of $J^\kappa / D$ of cofinality $(\theta_1, \theta_2),$ where $\theta_1, \theta_2$ are infinite, then it is induced by a cut of $J_*^\kappa/D.$
\end{lemma}
\begin{proof}
$(a)$ Suppose $(C_1, C_2)$ is a cut of $J_*^\kappa/D$ of cofinality $(\theta, \sigma),$ where $\theta, \sigma\leq \lambda,$ and let $\langle \langle f_\alpha/D: \alpha<\theta \rangle,  \langle g_\beta/D: \beta<\sigma \rangle  \rangle$ witness it. Assume the cut is filled in $J^\kappa/D$ by some $h\in J^\kappa.$ Then for all $\alpha<\theta, \beta<\sigma, f_\alpha <_D h <_D g_\beta,$ and hence $A_{\alpha, \beta}=\{i<\kappa: f_\alpha(i) <_J h(i) <_J g_\beta(i)  \}\in D.$ For $i<\kappa$ set
\begin{center}
$\Gamma_i=\{f_\alpha(i) <_{J_*} x <_{J_*} g_\beta(i): (\a, \beta)\in \theta\times \sigma$ such that $ i\in A_{\alpha,\beta} \}.$
\end{center}
$\Gamma_i$ is easily seen to be finitely satisfiable in $J$, hence also in $J_*$, so it is realized by some $h_*(i)\in J_*.$ Then $h_*\in J_*^\kappa/D,$ and for all $\alpha<\theta, \beta<\sigma$ we have $f_\alpha <_D h_* <_D g_\beta,$ a contradiction.

$(b)$ follows easily from the fact that $J_*$ is dense in $J$.
\end{proof}

\begin{theorem}
Suppose $\kappa$ is a measurable cardinal, and let $D$ be a $\kappa-$complete ultrafilter on $\kappa$. Then:

$(a)$ $\mathscr{C}_{\leq \kappa}(D)=\emptyset.$

$(b)$ $(\theta, \sigma)\notin \mathscr{C}(D),$ where $\theta<\kappa.$
\end{theorem}
\begin{proof}
$(a)$ follows from \cite{malliaris-shelah2} Claim 9.1\footnote{In \cite{malliaris-shelah2} Claim 9.1, it is proved that $(\theta,\sigma)\notin \mathscr{C}(D)$ when $\theta, \sigma<\kappa$ or $\theta=\sigma=\kappa$. The case $\theta<\sigma=\kappa$ can be proved similarly. }, and $(b)$ is \cite{malliaris-shelah2} Claim 10.3.
\end{proof}
We now give a generalization of Theorem 2.2.
\begin{theorem}
Suppose $\k_1>\aleph_0$ and $D$ is a  $\k_1-$complete (but not $\k_1^+-$complete; hence $\k_1$ is a measurable cardinal) ultrafilter on $\k_2$ and $(\theta_1, \theta_2)\in \mathscr{C}(D).$ Then $\theta_1, \theta_2 >\k_1.$
\end{theorem}
\begin{remark}
If $\k_1=\aleph_0,$ then $(\aleph_0, \lcf(\aleph_0, D)) \in  \mathscr{C}(D),$ and $\lcf(\aleph_0, D)$ may be $\aleph_1,$ where $\lcf(\aleph_0, D)$ denotes the lower cofinality \footnote{See \cite{shelah3}, Chapter VI for the definition of $\lcf(\aleph_0, D)$ and more information about it.}.
\end{remark}
\begin{proof}
Toward contradiction, assume that $\theta_1, \theta_2$ are regular, $\theta_1 \leq \k_1$ and $(\theta_1, \theta_2)\in \mathscr{C}(D)$ (recall that $\mathscr{C}(D)$ is symmetric, so we can assume w.l.o.g. that $\theta_1\leq \k_1$). Let $\l=\k_2+\theta_1+\theta_2$, let $J_*$ be a $\l^+-$saturated dense linear order, and suppose that $\langle  \langle f_\a/D : \a<\theta_1 \rangle, \langle g_\b/D: \b<\theta_2 \rangle\rangle  $ witnesses  $(\theta_1, \theta_2)\in \mathscr{C}(D)$, where $f_\a, g_\b \in J_*^{\k_2}.$ Let $\bar{f}/D=\langle f_\a/D : \a<\theta_1 \rangle,$ $\bar{g}/D=\langle g_\b/D : \b<\theta_2 \rangle$ and let $J$ be the completion of $J_*$. By Lemma 2.1, $\langle\bar{f}, \bar{g} \rangle$ also witnesses a cut of $J^{\k_2}/D.$

Let $\chi$ be a large enough regular cardinal such that $\bar{f}, \bar{g}\in H(\chi)$, and consider the structure $\mathfrak{A}_1=(H(\chi), \in)$ and let $\mathfrak{A}_2=\mathfrak{A}_1^{\k_2}/D.$ Also let $j: \mathfrak{A}_1  \rightarrow \mathfrak{A}_2$ be the canonical elementary embedding. Clearly $j$ is identity on $H(\k_1)$ and $crit(j)= \k_1$.

\begin{claim}
There exists $s\in \mathfrak{A}_2$ such that:

$\hspace{1.cm}$$(a)$ $\mathfrak{A}_2\models $`` $s$ is a function from $j(\theta_1)$ into $j(J_*)$'',

$\hspace{1.cm}$$(b)$ $\mathfrak{A}_2\models$`` $s(\a)=f_\a/D$'' for every $\a<\theta_1.$
\end{claim}
\begin{proof}
Define $h:\k_2 \rightarrow \mathfrak{A}_1$ by
\begin{center}
$h(\xi)= \langle f_\a(\xi): \a < \theta_1 \rangle.$
\end{center}
Let $s=h/D\in \mathfrak{A}_2.$ We show that $s$ is as required. First note that $s=j(h)(\k_1)$

It is clear that
\begin{center}
$s=j(h)(\k_1)= \langle j(\bar{f})_\a(\k_1): \a< j(\theta_1) \rangle.$
\end{center}
But for $\a<\theta_1,$ we have $j(\bar{f})_\a(\k_1)=j(\bar{f})_{j(\a)}(\k_1)=j(f_\a)(\k_1)=f_\a/D.$
The result follows immediately.
\end{proof}
So we have the following:
\begin{claim}
There exist $s\in \mathfrak{A}_2$ and $\k_*$ such that:

$\hspace{1.cm}$$(a)$ $\mathfrak{A}_2\models $`` $s$ is a function with domain $\k_*$'',

$\hspace{1.cm}$$(b)$ $\k_*$ is \footnote{In fact it is easily seen that $\k_*=\theta_1.$} $j(\theta_1)$ if $\theta_1<\k_1$ and the least upper bound of $\{j(\a): \a<\theta_1\}$ if $\theta_1=\k_1,$

$\hspace{1.cm}$$(c)$ $\mathfrak{A}_2\models$`` $s(\a)=f_\a/D$'' for every $\a<\theta_1.$
\end{claim}
Fix $s$ as in Claim 2.6.
\begin{claim}
If $A_2 \subseteq \mathfrak{A}_2, |A_2| \leq \l$ and $b\in A_2 \Rightarrow \mathfrak{A}_2\models$`` $ b\in j(J_*)$ and $s(\a) <_{j(J_*)} b$'' for all $\a<\theta_1$, then for some $b_*\in j(J),$ we have
\begin{center}
$b\in A_2$ and $\a<\theta_1 \Rightarrow \mathfrak{A}_2\models$`` $f_\a/D <_{j(J)} b_* \leq_{j(J)} b$''.
\end{center}
\end{claim}
\begin{proof}
Let $b_*\in j(J)$ be such that $\mathfrak{A}_2\models$``$b_*$ is the $<_{j(J)}-$least upper bound of $s(\a), \a<\theta_1$''. Note that such a $b_*$ exists as $\mathfrak{A}_2\models$``$j(J)$ is a complete dense linear order'' and $s\in \mathfrak{A}_2$. It is easily seen that  for $b\in A_2$ and $\a<\theta_1$ we have  $\mathfrak{A}_2\models$`` $f_\a/D <_{j(J)} b_* \leq_{j(J)} b$''.
\end{proof}
Let $A_2=\{g_\b/D : \b<\theta_2\}.$ As $\theta_2\leq\lambda,$ and for $\b<\theta_2, g_\b/D\in j(J_*),$ so we can apply Claim 2.6 to find some $b_*\in j(J)$ such that for all $\a<\theta_1, \b<\theta_2, \mathfrak{A}_2\models$`` $f_\a/D <_{j(J)} b_* \leq_{j(J)} g_\b/D$''. Let $h\in J^\k$ be such that $b_*=h/D.$ Then
\begin{center}
$\a<\theta_1, \b<\theta_2 \Rightarrow \mathfrak{A}_2\models$`` $f_\a/D <_{j(J)} h/D \leq_{j(J)} g_{\b+1}/D <_{j(J)} g_\b/D$''.
\end{center}
It follows that the cut $\langle \bar{f}, \bar{g} \rangle$ is filled in $J^\k/D,$ which is in contradiction with Lemma 2.1$(a)$. The theorem follows.
\end{proof}

The next theorem is implicit in \cite{malliaris-shelah2} Theorem 11.3. We give a proof for completeness.
\begin{theorem}
Suppose $D$ is an ultrafilter on $\kappa$ and $\theta>\kappa$ is weakly compact. Then $(\theta, \theta)\notin \mathscr{C}(D).$
\end{theorem}
\begin{proof}
Suppose not. Let $J$ be a $\theta^+-$saturated dense linear order and suppose $\langle  \langle f_\a/D: \a<\theta  \rangle, \langle g_\a/D: \a<\theta  \rangle \rangle$ witnesses $(\theta, \theta)\in \mathscr{C}(D).$ For $\a <\b <\theta$ we have
\begin{center}
$A_{\a,\b}=\{i<\k: f_\a(i) <_J f_\b(i) <_J g_\b(i) <_J g_\a(i)   \}\in D.$
\end{center}
Define $d: [\theta]^2 \rightarrow D$ by $d(\a,\b)=A_{\a,\b}.$ Since $2^\k <\theta$ and $\theta$ is weakly compact, we can find $X\in [\theta]^\theta$ and $A_*\in D$ such that for all $\a<\b$ in $X, A_{\a,\b}=A_*.$ For $i\in A_*$ consider the type
\begin{center}
$\Gamma_i=\{f_\a(i) <_J x <_J g_\a(i): \a\in X  \}.$
\end{center}
\begin{claim}
$\Gamma_i$ is finitely satisfiable in $J$.
\end{claim}
\begin{proof}
Let $\a_0 < ... <\a_n$ be in $X$. Then
\begin{center}
$f_{\a_0}(i) <_J ... <_J f_{\a_n}(i) <_J g_{\a_n}(i) <_J ... <_J g_{\a_0}(i).$
\end{center}
So it suffices to pick some element in $(f_{\a_n}(i), g_{\a_n}(i))_J,$ which is possible as $J$ is dense.
\end{proof}
As $J$ is $\theta^+-$saturated, there is $h(i)\in J$ realizing $\Gamma_i.$ So $h(i)$ is defined in this way for every $i \in A_*.$
Let $h: A_* \to J$ be the resulting function. Extend $h$ to a function in $J^\k.$ Then
\begin{center}
$\forall i\in A_*, \forall \a\in X, f_\a(i) <_J h(i) <_J g_\a(i) \Rightarrow \forall\a\in X, f_\a <_D h <_D g_\a.$
\end{center}
Since $X$ is unbounded in $\theta,$ we have $\forall \a<\theta, f_\a <_D h <_D g_\a,$ and we get a contradiction.
\end{proof}
We may note that the use of the partition relation $\theta \to (\theta)^2$ in the above proof is optimal in the following sense: If $\theta$ is strongly inaccessible but not weakly compact, then $\theta \to (\theta, \alpha)^2$ for every $\alpha < \theta,$ yet, this is not enough for excluding the pair $(\theta, \theta)$,
as proved in 3.4.$($b$)$.

The next theorem generalizes the above result.
\begin{theorem}
Suppose that $D$ is an ultrafilter on $\k,$
 $\k<\mu\leq\l, \theta,$ where
 $\mu$ is a supercompact cardinal and $\l, \theta$ are regular.
Then $(\l,\theta)\notin \mathscr{C}(D).$
\end{theorem}
\begin{proof}
Suppose not. Let $J$ be a $(\lambda+\theta)^+-$saturated dense linear order, $J \subseteq Ord,$ and let $\langle \langle f_\alpha/D: \alpha<\lambda \rangle,  \langle g_\gamma/D: \gamma<\theta \rangle  \rangle$ witness a pre-cut in $J^\k/D,$ which is a pre-cut in $J_*^\k/D,$ for each linear order $J_* \supseteq J.$

Let $\bar{f}/D=\langle f_\alpha/D: \alpha<\lambda \rangle$ and $\bar{g}/D= \langle g_\gamma/D: \gamma<\theta \rangle.$ Let $\eta =(\l+\theta), U$ be a normal measure on $P_\mu(\eta)$ and let $j: V \rightarrow M\simeq Ult(V, U)$ be the corresponding elementary embedding. So we have $crit(j)=\mu$ and $M^{\eta} \subseteq M.$ It follows that $j(\k)=\k$ and $j(D)=D.$

Note that $j[J]$ is also a $(\lambda+\theta)^+-$saturated dense linear order and for $f\in J^\k, j(f)=j[f]\in j[J]^\k,$ as $|f|=\k < \eta=crit(j).$ It follows that:

$\hspace{1.5cm}$ ``$\langle \langle j(f_\alpha)/D: \alpha<\lambda \rangle,  \langle j(g_\gamma)/D: \gamma<\theta \rangle  \rangle$  witnesses a pre-cut in

$(*)_1$$\hspace{1.cm}$  $j[J]^\k/D,$ which is also a pre-cut in $J_*^\k/D,$ for each linear order

$\hspace{1.5cm}$ $J_* \supseteq j[J]$''.

Also, as $j$ is an elementary embedding, the following hold in $M$:

$\hspace{1.5cm}$ ``$j(J)$ is a $j((\lambda+\theta)^+)-$saturated dense linear order, $j(J) \supseteq j[J],$

$(*)_2$$\hspace{.9cm}$ and $\langle \langle j(\bar{f})_\alpha/D: \alpha<j(\lambda) \rangle,  \langle j(\bar{g})_\gamma/D: \gamma<j(\theta) \rangle  \rangle$ witnesses a pre-

$\hspace{1.5cm}$ cut in $j(J)^\k/D$''.

On the other hand  $M^\eta \subseteq M$, so the sequences $\langle j(f_\alpha)/D: \alpha<\lambda \rangle$ and  $\langle j(g_\gamma)/D: \gamma<\theta \rangle$ are in $M$, and by $(*)_1$, the following holds in $M$:

$(*)_3$ $\hspace{.7cm}$ ``$\langle \langle j(f_\alpha)/D: \alpha<\lambda \rangle,  \langle j(g_\gamma)/D: \gamma<\theta \rangle  \rangle$  witnesses a pre-cut in $j(J)^\k/D$''.

We now derive a contradiction from  $(*)_2$ and $(*)_3.$

Assume that $\theta\geq \l.$ Note that $\sup\{j(\gamma): \gamma< \theta\} < j(\theta),$ so pick $\delta$ such that $\sup\{j(\gamma): \gamma< \theta\} < \delta < j(\theta).$ Consider $j(\bar{f})_\delta \in j(J)^\k.$ Then by $(*)_2$, for all $\a<\l$ and $\gamma<\theta$
\begin{center}
$j(f_\a)=j(\bar{f})_{j(\a)} <_D j(\bar{f})_\delta <_D j(\bar{g})_{j(\gamma)}=j(g_\gamma).$
\end{center}
This contradicts $(*)_3$. The theorem follows.
\end{proof}
In fact we can weaken the above assumptions, as  shown in the next theorem. We have given the above proof, as it appears in later sections of the paper, where the methods of the proof of Theorem 2.11 are not applicable (see Theorem 5.10 and Claim 5.23).
\begin{theorem}
Suppose that $D$ is an ultrafilter on $\kappa$, $\kappa<\mu \leq \lambda, \theta,$ where $\mu$ is a strongly compact cardinal and $\lambda, \theta$ are regular. Then $(\lambda, \theta)\notin \mathscr{C}(D).$
\end{theorem}
\begin{proof}
Suppose not. Let $J$ be a $(\lambda+\theta)^+-$saturated dense linear order, and let $\langle \langle f_\alpha/D: \alpha<\lambda \rangle,  \langle g_\gamma/D: \gamma<\theta \rangle  \rangle$ witness $(\lambda, \theta) \in \mathscr{C}(D).$ For $\alpha_1 <\alpha_2 <\lambda$ and $\gamma_1<\gamma_2 <\theta$ set
\begin{center}
$A_{\alpha_1,\alpha_2,\gamma_1,\gamma_2}=\{i<\kappa: f_{\alpha_1}(i)<_J f_{\alpha_2}(i) <_J g_{\gamma_2}(i) <_J g_{\gamma_1}(i) \} \in D.$
\end{center}
Let $E$ be a $\mu-$complete uniform fine ultrafilter on $\lambda\times \theta$ such that
\begin{center}
$(\a, \gamma)\in \lambda\times \theta \Rightarrow \{(\bar{\a}, \bar{\gamma})\in \lambda\times \theta: \a<\bar{\a}, \gamma <\bar{\gamma}   \} \in E.$
\end{center}
We can find such ultrafilter $E,$ as the cardinals $\l, \theta$ are regular $\geq \mu$ and $\mu$ is a strongly compact cardinal. As $E$ is $\mu-$complete and $|D|=2^\k <\mu,$ for each $(\alpha,\gamma) \in \lambda\times\theta,$ there exists a unique set $A_{\alpha,\gamma}\in D$ such that
\begin{center}
$X_{\alpha,\gamma}=\{ (\bar{\alpha}, \bar{\gamma})\in \lambda\times\theta:  \alpha<\bar{\alpha}, \gamma<\bar{\gamma}, A_{\alpha,\bar{\alpha},\gamma, \bar{\gamma}}=A_{\alpha,\gamma}  \} \in E.$
\end{center}
For $i<\kappa$ consider the type
\begin{center}
$\Gamma_i=\{ f_\alpha(i) <_J x <_J g_\gamma(i): \alpha <\lambda$ and $\gamma<\theta$ are such that $i\in A_{\alpha,\gamma} \}.$
\end{center}
\begin{claim}
$\Gamma_i$ is finitely satisfiable in $J$.
\end{claim}
\begin{proof}
Suppose $\alpha_1, ..., \alpha_n <\lambda, \gamma_1, ..., \gamma_n <\theta$ and $i\in A_{\alpha_1,\gamma_1} \cap ... \cap A_{\alpha_n, \gamma_n}.$ Choose $\alpha>\alpha_1, ..., \alpha_n$ and $\gamma>\gamma_1, ..., \gamma_n$ such that for all $1\leq l\leq n, A_{\alpha_l,\alpha,\gamma_l,\gamma}=A_{\alpha_l,\gamma_l}.$ Then
\begin{center}
$f_{\alpha_l}(i) <_J f_\alpha(i) <_J g_\gamma(i) <_J g_{\gamma_l}(i).$
\end{center}
So it suffices to take some $x\in (f_\alpha(i), g_\gamma(i))_J.$
\end{proof}
It follows that $\Gamma_i$ is realized by some $h(i).$ As usual, extend $h$ to a function on $\kappa.$ Then
\begin{center}
$\forall i<\kappa, \forall (\alpha,\gamma)\in \lambda\times\theta (i\in A_{\alpha,\gamma} \Rightarrow f_\alpha(i) <_J h(i) <_J g_\gamma(i)).$
\end{center}
Thus for $\alpha<\lambda, \gamma<\theta,$
\begin{center}
$f_\alpha <_D h <_D g_\gamma,$
\end{center}
and we get a contradiction.
\end{proof}

\begin{definition} (\cite{erdos})
$\begin{pmatrix}\k \\
\l\end{pmatrix}\to\begin{pmatrix}\mu\\\nu\end{pmatrix}^{1,1}_\rho$ means: if $d: \k\times\l \rightarrow \rho,$ then for some $A\subseteq \k$ of order type $\mu$ and some $B\subseteq \l$ of order type $\nu, d \upharpoonright A\times B$ is constant.
\end{definition}
\begin{theorem}
Suppose  $\begin{pmatrix}\l \\
\theta\end{pmatrix}\to\begin{pmatrix}\l\\\theta\end{pmatrix}^{1,1}_{2^\k},$ and $D$ is an ultrafilter on $\k.$ Then $(\l, \theta)\notin \mathscr{C}(D).$
\end{theorem}
\begin{proof}
Suppose not. Let $J$ be some $(\l+\theta)^+-$saturated dense linear order and let $\langle \langle f_\alpha/D: \alpha<\l \rangle,  \langle g_\beta/D: \beta<\theta \rangle  \rangle$ witness $(\l, \theta) \in \mathscr{C}(D),$ where $f_\alpha, g_\beta \in J^\k.$ For $\alpha<\l, \beta<\theta$ set
\begin{center}
$A_{\alpha, \beta}=\{i<\k: f_\alpha(i) <_J g_\beta(i)    \}\in D.$
\end{center}
Define $d: \l\times\theta \rightarrow D$ by $d(\a,\b)=A_{\a,\b}.$ By our assumption, there are $A\in [\l]^\l$, $B\in [\theta]^\theta$ and $A_*\in D$
such that for all $\a\in A, \b\in B, A_{\a,\b}=A_*.$ For $i\in A_*$ set
\begin{center}
$\Gamma_i=\{f_\alpha(i) <_J x <_J g_\b(i): \a\in A, \b\in B\}.$
\end{center}
$\Gamma_i$ is easily seen to be finitely satisfiable, and hence it is realized by some $h(i).$ Extend $h$ to a function in $J^\k.$
\begin{claim}
For all $\a<\l, \b<\theta, f_\a<_D h <_D g_\b.$
\end{claim}
\begin{proof}
Let $\a<\l, \b<\theta.$ Choose $\a^*\in A, \b^* \in B$ such that $\a<\a^*, \b<\b^*.$ Then
\begin{center}
$\{i<\k: f_{\a^*}(i) <_J h(i) <_J g_{\b^*}(i)  \} = A_*\in D,$
\end{center}
so $f_\a <_D f_{\a^*} <_D h <_D g_{\b^*} <_D g_\b.$
\end{proof}
We get a contradiction, and the theorem follows.
\end{proof}

\begin{theorem}
Suppose that $D$ is an ultrafilter on $\kappa$ and $\theta, \lambda$ are regular cardinals such that $\theta^\kappa <\lambda.$ Then $(\lambda, \theta) \notin \mathscr{C}(D).$
\end{theorem}
\begin{proof}
Suppose not. Let $J_*$ be a $\lambda^+-$saturated dense linear order and suppose that $(C_1, C_2)$ is a cut of $J_*^\kappa/D$ with cofinality $(\lambda, \theta).$ Also let $J$ be the completion of $J_*$. Let $\langle \langle f_\alpha/D: \alpha<\lambda \rangle,  \langle g_\gamma/D: \gamma<\theta \rangle  \rangle$ witness $cf(C_1, C_2)=(\lambda, \theta),$ where $f_\alpha, g_\gamma \in J_*^\kappa, \alpha<\lambda, \gamma<\theta.$ By $\lambda^+-$saturation of $J_*$ and the remarks after Definition 1.1, we can find $s_{-\infty}, s_{+\infty} \in J_*$ such that for all $\alpha<\lambda, \gamma<\theta$ and $i<\kappa,$
\begin{center}
$s_{-\infty} <_{J_*} f_\alpha(i), g_\gamma(i) <_{J_*} s_{+\infty}.$
\end{center}
Let $I=\{g_\gamma(i): \gamma<\theta, i<\kappa  \}\cup \{s_{-\infty}, s_{+\infty}   \}.$ Then $|I|^\kappa=(\theta+\k)^\k  =\theta^\k <\lambda$ and $g_\gamma \in I^\kappa$ for all $\gamma<\theta.$
\begin{claim}
There is $\beta_*<\lambda$ such that for all $g\in I^\kappa$ and $\beta \in [\beta_*, \lambda):$
\begin{itemize}
\item $g <_D f_\beta \Leftrightarrow g <_D f_{\beta_*},$
\item $f_\beta <_D g \Leftrightarrow f_{\beta_*} <_D g.$
\end{itemize}
\end{claim}
\begin{proof}
Suppose not. So we can find an increasing sequence $\langle  \beta_\xi: \xi<\lambda  \rangle$ of ordinals $<\lambda$ and a sequence $\langle  g_\xi: \xi<\lambda  \rangle$
of elements of $I^\kappa$ such that for all $\xi<\lambda, f_{\beta_\xi} \leq_D g_\xi <_D f_{\beta_{\xi+1}}.$

It follows that for $\xi < \zeta <\lambda, g_\xi \neq g_\zeta,$ hence $|I|^\kappa \geq \lambda,$ a contradiction.
\end{proof}
Fix $\beta_*$ as above. We define a function $g_* \in J^\kappa$ as follows: Let $i<\kappa.$ Consider the set
\begin{center}
$I_i=\{ t\in I: f_{\beta_*}(i)\leq_{J_*} t \}.$
\end{center}
We have $s_{+\infty} \in I_i$ and $I_i$ is bounded from below , so as $J$ is complete,
\begin{center}
$g_*(i)=$the $<_J-$greatest lower bound of $I_i$
\end{center}
is well-defined, so $g_*\in J^\kappa.$ It is clear that for all $i<\kappa, f_{\beta_*}(i) \leq_J g_*(i)$ so

$(*)$ $\hspace{5.cm}$ $f_{\beta_*} \leq_D g_*.$

We show that for all $\alpha<\lambda, \gamma<\theta,$ $f_\alpha \leq_D g_* \leq_D g_\gamma,$ which will give us the desired contradiction.
\begin{claim}
For all $\alpha<\lambda, f_\alpha \leq_D g_*.$
\end{claim}
\begin{proof}
Since $\langle f_\alpha/D: \alpha<\lambda \rangle$ is increasing, we may suppose that $\alpha \in [\beta_*,\lambda).$ Suppose on the contrary that $g_* <_D f_\alpha.$ So
\begin{center}
$u=\{i<\kappa: g_*(i) <_J f_\alpha(i)  \} \in D.$
\end{center}
For $i\in u, g_*(i) <_J f_\alpha(i),$ so by our definition of $g_*(i),$ we can find $g(i)\in [g_*(i), f_\alpha(i))_{J_*}\cap I.$ For $i\in \kappa\setminus u$ set $g(i)=s_{+\infty}.$ Then $g\in I^\kappa$ and $g_* \leq_D g <_D f_\alpha.$ By $(*)$, $f_{\beta_*} \leq_D g <_D f_\alpha$ which is in contradiction with Claim 2.17.
\end{proof}
\begin{claim}
For all $\gamma<\theta, g_* \leq_D g_\gamma.$
\end{claim}
\begin{proof}
Fix any ordinal $\gamma < \theta$. Since $f_{\beta_*} <_D g_\gamma,$ we have
\begin{center}
$u=\{i<\kappa: f_{\beta_*}(i) <_J g_\gamma(i)  \} \in D.$
\end{center}
Hence for $i\in u, g_\gamma(i) \in I_i$ and it follows that $g_*(i) \leq_{J_*} g_\gamma(i).$ So $g_* \leq_D g_\gamma.$
\end{proof}
Since $g_* \neq_D f_\alpha$ for every
$\alpha < \lambda$ (as $\langle f_\alpha / D: \alpha < \lambda      \rangle$ is increasing) and $g_* \neq_D g_\gamma$ for every
 $\gamma < \theta$ (by a similar argument), we have  $\forall \alpha<\lambda \forall \gamma<\theta, f_\alpha \leq_D g_* \leq_D g_\gamma,$   and we get a contradiction. The theorem follows.
\end{proof}

Recall that the \emph{singular cardinals hypothesis} ($SCH$) says that if $2^{cf(\k)}<\k,$ then $\k^{cf(\k)}=\k^+.$ It follows from $SCH$ that if $\theta < \lambda$ are regular cardinals and $ \l> 2^\k,$ then $\theta^\k <\l$ (see \cite{jech}, Theorem 5.22). The following corollary follows from Theorem 2.16.
\begin{corollary}
$(a)$ ($GCH$) Suppose that $D$ is an ultrafilter on $\kappa$ and $\theta\leq \lambda$ are regular cardinals such that $\l>\kappa^+.$ If $(\lambda, \theta)\in \mathscr{C}(D),$ then $\lambda=\theta.$

$(b)$ ($SCH$) Suppose that $D$ is an ultrafilter on $\kappa$ and $\theta\leq \lambda$ are regular cardinals such that $\l> 2^\k.$ If $(\lambda, \theta)\in \mathscr{C}(D),$ then $\lambda=\theta.$
\end{corollary}
The next corollary follows from Theorems 2.11 and 2.16:
\begin{corollary}
Suppose that $D$ is an ultrafilter on $\kappa$, $\mu>\k$ is strongly compact and $(\theta, \sigma)\in \mathscr{C}(D).$ Then $\theta, \sigma <\mu,$ in particular $\mathscr{C}(D)$ is a set.
\end{corollary}

\begin{theorem}
Suppose $D$ is an ultrafilter on $\kappa,$ $\kappa <\lambda=cf(\lambda),$ and suppose that $(*)^{\partial, \bar{n}}_{\lambda, \theta}$ holds for $\theta=2^\kappa, \partial\leq \aleph_0, \bar{n} \leq \omega,$ where:

$\hspace{1.5cm}$ If $c:[\lambda]^2 \rightarrow \theta,$ then there are $u\subseteq \theta, |u|< 1+ \partial$ and $S\in [\lambda]^\lambda$ such that

$(*)^{\partial, \bar{n}}_{\lambda, \theta}:$$\hspace{.4cm}$ if $\alpha<\beta$ are in $S$, then for some $n< 1+\bar{n}$ and $\gamma_0<\gamma_1 <...<\gamma_n$ we

$\hspace{1.5cm}$  have $\gamma_0=\alpha, \gamma_n=\beta$ and for $l<n, c\{\gamma_l,\gamma_{l+1}\}\in u.$

Then $(\lambda,\lambda)\notin \mathscr{C}(D).$
\end{theorem}
\begin{remark} (\cite{shelah4}, Remark 2.3)
If $\kappa<\mu\leq\lambda=cf(\lambda), (\forall \a<\lambda) |\a|^\k <\l$ and $\mu$ is a strongly compact cardinal, then the following holds:

$\hspace{1.cm}$ If $c:[\lambda]^2 \rightarrow \k,$ then there are $i,j<\k$ and $S\in [\lambda]^\lambda$ such that for all

$\hspace{1.cm}$ $\alpha<\beta$ in $S$, there is $\alpha<\gamma<\beta$ such that $c\{\alpha,\gamma\}=i$ and $c\{\gamma,\beta\}=j.$

Hence $(*)^{2,2}_{\lambda,\k}$ holds.
\end{remark}
The following example shows that we can not remove the assumption $(\forall \a<\lambda) |\a|^\k <\l$ from Remark 2.23.
\begin{example}
Suppose that:

$\hspace{1.cm}$$(a)$ $\l=cf(\l) >\mu > \theta=cf(\mu),$

$\hspace{1.cm}$$(b)$ $\bar{\l}=\langle  \l_i: i<\theta \rangle $ is an increasing  unbounded sequence of regular cardinals in

$\hspace{1.5cm}$$(\theta, \mu),$

$\hspace{1.cm}$$(c)$ $D$ is an ultrafilter on $\theta,$

$\hspace{1.cm}$$(d)$ $\l=tcf(\prod_{i<\theta}\l_i, <_D),$ as witnessed \footnote{By \cite{shelah-g}, if $\l=\mu^+,$ then there exist a sequence $\bar{\l}$ as in $(b)$ and an ultrafilter $D$ on $\theta,$ such that $D$  contains all co-bounded subsets of $\theta$ and $\l=tcf(\prod_{i<\theta}\l_i, <_D)$.} by the scale $\bar{f}=\langle f_\a: \a<\l  \rangle.$

$\hspace{1.cm}$$(e)$ $c: [\l]^2 \rightarrow D$ is defined such that for $\a<\b<\l, f_\a \upharpoonright c\{\a,\b\} < f_\b \upharpoonright c\{\a,\b\}.$

Then for every $\mathcal{U}\in [\l]^\l$ and every finite $u \subseteq D$ there are some $\a<\b$ in $\mathcal{U}$ such that for no $\a=\gamma_0 < \gamma_1 < ... <\gamma_n=\b$ do we have $l<n \Rightarrow c\{\gamma_l, \gamma_{l+1}\}\in u.$

To see this, suppose that the claim fails. Pick $\xi\in \bigcap\{ A: A\in u\}$. Then for all $\a<\b$ in $\mathcal{U},$ we can find $\a=\gamma_0 < \gamma_1 < ... <\gamma_n=\b$ such that $l<n \Rightarrow c\{\gamma_l, \gamma_{l+1}\}\in u.$ But then $\xi\in \bigcap\{c\{\gamma_l, \gamma_{l+1}\}: l<n\},$ and hence
\begin{center}
 $f_\a(\xi)=f_{\gamma_0}(\xi) <f_{\gamma_1}(\xi)< ...< f_{\gamma_{n-1}}(\xi) < f_{\gamma_n}(\xi)=f_\b(\xi).$
\end{center}
Thus the sequence $\langle f_\a(\xi):\xi\in \mathcal{U} \rangle$ is a strictly increasing sequence of ordinals in $\l_\xi.$ But $\l_\xi < \l=|\mathcal{U}|,$ and we get a contradiction.
\end{example}

\begin{proof} (of Theorem 2.22).
Suppose not. Let $J$ be a $\lambda^+-$saturated dense linear order, and let  $\langle \langle f_\alpha/D: \alpha<\lambda \rangle,  \langle g_\a/D: \a<\lambda \rangle  \rangle$ witness $(\lambda, \lambda) \in \mathscr{C}(D).$ For $\alpha<\gamma<\lambda$ set
\begin{center}
$A_{\alpha,\gamma}=\{i<\kappa: f_\alpha(i) <_J f_\gamma(i) <_J g_\gamma(i) <_J g_\alpha(i) \} \in D.$
\end{center}
Define $c:[\lambda]^2  \rightarrow D$ by
\begin{center}
$c\{\alpha,\gamma \}=A_{\alpha,\gamma}.$
\end{center}
By $(*)^{\partial, \bar{n}}_{\lambda, 2^\kappa}$ we can find a finite set $u\subseteq D$ and a set $S\in[\lambda]^\lambda$ such that
if $\alpha<\beta$ are in $S$, then for some $n< 1+\bar{n}$ and $\alpha=\gamma_0< ... <\gamma_n=\beta$ we have $l<n \Rightarrow A_{\gamma_l, \gamma_{l+1}} \in u.$ Since
$u$ is finite, $A=\bigcap u$ belongs to $D$.
It follows immediately that for $\alpha<\beta$ in $S$ and $i\in A,$ we have

$(*)$ $\hspace{4.cm}$$f_\alpha(i) <_J f_\beta(i) <_J g_\beta(i) <_J g_\alpha(i).$

For $i\in A$ set
\begin{center}
$\Gamma_i=\{ f_\alpha(i) <_J x <_J g_\alpha(i): \alpha\in S  \}.$
\end{center}
By$(*)$, $\Gamma_i$ is finitely satisfiable, so it is realized by some $h(i)\in J.$ Then $h\in J^\kappa,$ and for all $\alpha<\lambda, f_\alpha<_D h <_D g_\alpha,$ and we get a contradiction.
\end{proof}

\section{on $\mathscr{C}(D)$ being large}
In this section we show that under some extra set theoretic assumptions, $\mathscr{C}(D)$ can be large. In particular, we show that if $V=L$, then $\mathscr{C}(D)$ can be a proper class. The following is proved in \cite{malliaris-shelah2}:
\begin{theorem} (\cite{malliaris-shelah2} Claim 10.1)
Suppose $\k$ is a measurable cardinal, and $D$ is a normal measure on $\kappa.$ Then $(\k^+, \k^+) \in \mathscr{C}(D).$
\end{theorem}
We can use the method of the proof of Theorem 2.3, to remove the normality assumption from the above theorem. More precisely we have the following:
\begin{theorem}
Suppose $\k$ is a measurable cardinal, and $D$ is a uniform $\k-$complete ultrafilter on $\kappa.$ Then $(\k^+, \k^+) \in \mathscr{C}(D).$
\end{theorem}

We now state and prove the main result of this section.
\begin{theorem}
Suppose that:

$\hspace{1.cm}$$(a)$ $D$ is a uniform ultrafilter on $\kappa,$

$\hspace{1.cm}$$(b)$ We have  $2^\kappa <\lambda=cf(\lambda)$,

$\hspace{1.cm}$$(c)$ The sequence $\bar{C}=\langle C_\alpha: \alpha<\lambda \rangle$ satisfies

$\hspace{1.5cm}$$(\alpha)$ $C_\alpha\subseteq\alpha,$

$\hspace{1.5cm}$$(\beta)$ $lim(\alpha) \Rightarrow sup(C_\alpha)=\alpha,$

$\hspace{1.5cm}$$(\gamma)$ $\beta\in C_\alpha \Rightarrow C_\beta=C_\alpha\cap\beta,$

$\hspace{1.5cm}$$(\delta)$ If $\alpha$ is a successor ordinal, then either $C_\alpha$ is empty, or has a last element.

$\hspace{1.cm}$$(d)$ $S=\{\delta<\lambda: otp(C_\delta)=\kappa, \delta\notin \bigcup_{\alpha<\lambda}C_\alpha \}$ is a stationary subset \footnote{Note that $S$ is necessarily non-reflecting.}  of $\lambda$,

$\hspace{1.cm}$$(e)$ There exists a sequence $\langle (a_\delta, \xi_\delta): \delta\in S  \rangle$ such that:

$\hspace{1.5cm}$$(\alpha)$ We have  $a_\delta\subseteq nacc(C_\delta)$,

$\hspace{1.5cm}$$(\beta)$ $otp(a_\delta)=\kappa,$

$\hspace{1.5cm}$$(\gamma)$ $\xi_\delta<2^\kappa,$

$\hspace{1.5cm}$$(\delta)$ For every $h: \lambda\rightarrow 2^\kappa,$ there is some $\delta\in S$ such that $h\upharpoonright a_\delta$ is constantly $\xi_\delta.$

Then $(\lambda,\lambda)\in \mathscr{C}(D).$
\end{theorem}
\begin{proof}
Let $J$ be a $\lambda^+-$saturated dense linear order. We shall choose the functions $f_\alpha, g_\alpha\in J^\kappa, \alpha<\lambda$ such that $\langle  \langle f_\alpha/D:\alpha<\lambda \rangle, \langle g_\alpha/D: \alpha<\lambda \rangle   \rangle$ witnesses a cut of $J^\kappa/D$ of cofinality $(\lambda, \lambda).$ More specifically we shall choose the functions $f_\alpha, g_\alpha$ such that:

$\hspace{1.cm}$$(a)$ $\forall i<\kappa, f_\alpha(i)<_J g_\alpha(i),$

$\hspace{1.cm}$$(b)$ for $\beta<\alpha<\lambda, f_\beta <_D f_\alpha <_D g_\alpha <_D g_\beta,$

$\hspace{1.cm}$$(c)$ if $\alpha\notin S$ and $\beta\in C_\alpha,$ then  $\forall i<\kappa, f_\beta(i) <_J f_\alpha(i) <_J g_\alpha(i) <_J g_\beta(i),$

$\hspace{1.cm}$$(d)$ if $\delta\in S,$ and $\langle \alpha_{\delta,i}: i<\kappa \rangle$ enumerates $a_{\delta}$ in increasing order, then
\begin{center}
$\forall i<\kappa, f_\delta(i)=f_{\alpha_{\delta,i}}(i)$ and  $g_\delta(i)= f_{\alpha_{\delta, i+1}}(i).$
\end{center}
We do the construction by induction on $\alpha<\lambda.$

{\bf Case 1. $\alpha=0$:}
Let $f_0, g_0\in J^\kappa$ be such that for all $i<\kappa, f_0(i) <_J g_0(i).$

{\bf Case 2. $\alpha=\gamma+1$ is a successor ordinal:} By our assumption $\langle  \langle f_\xi:\xi\leq\gamma \rangle, \langle g_\xi: \xi\leq \gamma \rangle   \rangle$
is defined. We define $f_\alpha, g_\alpha.$
\begin{itemize}
\item {\bf Subcase 2.1. $C_\alpha=\emptyset:$} Then as $J$ is $\l^+-$saturated dense linear order, we can choose $f_\alpha, g_\alpha\in J^\kappa$ such that:

$\hspace{1.cm}$$(\alpha)$ for all $i<\kappa, f_\gamma(i) <_J f_\alpha(i) <_J g_\alpha(i) <_J g_\gamma(i).$

It is easily seen that  $(a)-(c)$ are satisfied, and  $(d)$ is vacuous.
\item {\bf Subcase 2.2. $C_\alpha$ has a last element $\delta$:}  Then we take $f_\alpha, g_\alpha\in J^\kappa$ as above with the additional property:

$\hspace{1.cm}$$(\beta)$ for all $i<\kappa, f_\delta(i) <_J f_\alpha(i) <_J g_\alpha(i) <_J g_\delta(i).$

Again it is easily seen that $(a)-(c)$ are satisfied, and  $(d)$ is vacuous.
\end{itemize}

{\bf Case 3. $\alpha$ is a limit ordinal $\alpha\notin S$:} We take $f_\alpha, g_\alpha\in J^\kappa$ such that:

$\hspace{1.cm}$$(\alpha)$ $\forall i<\kappa, f_\alpha(i) <_J g_\alpha(i),$

$\hspace{1.cm}$$(\beta)$ for $\beta<\alpha, f_\beta <_D f_\alpha <_D g_\alpha <_D g_\beta,$

$\hspace{1.cm}$$(\gamma)$ if $\beta\in C_\alpha,$ then $\forall i<\kappa, f_\beta(i) <_J f_\alpha(i) <_J g_\alpha(i) <_J g_\beta(i).$

For $i<\kappa,$ consider the type
\begin{center}
$\Gamma_i=\{f_\beta(i)<_J x <_J y <_J g_\beta(i): \beta\in C_\alpha     \}.$
\end{center}
By clause $(c)$ and using the induction hypothesis, $\Gamma_i$ is finitely satisfiable, so it is realized by some $f_\alpha(i) <_J g_\alpha(i).$ Clearly $(a)-(c)$ are satisfied and there is nothing to do with $(d)$.

{\bf Case 4. $\delta\in S$:} Then we define $f_\delta, g_\delta$ as in $(d),$ so $(d)$ holds. We must show that $(a)-(c)$ are also satisfied. Sice $(c)$
 is vacuous  in this case, it suffices to consider $(a)$ and $(b)$. For $(a),$ we have for any $i<\kappa,$
\begin{center}
$f_\delta(i)=f_{\alpha_{\delta,i}}(i) <_J f_{\alpha_{\delta,i+1}}(i)=g_\delta(i), $
\end{center}
by $(c),$ and the fact that $\alpha_{\delta,i} < \alpha_{\delta,i+1}$ are in $a_\delta \subseteq C_\delta,$ so $\alpha_{\delta,i} \in C_{\alpha_{\delta,i+1}}$
and hence $\alpha_{\delta,i} \notin S,$ so $(c)$ applies.

We check $(b).$ So suppose that $\beta < \delta.$ Since $\delta=sup_{j<\kappa}\alpha_{\delta,j},$ we can find $j<\kappa$ such that $\beta<\alpha_{\delta,j}.$ Then $f_\beta <_D f_{\alpha_{\delta,j}} <_D g_{\alpha_{\delta,j}} <_D g_\beta$ and hence
\begin{center}
$u=\{ i<\kappa: f_\beta(i) <_J f_{\alpha_{\delta,j}}(i) <_J g_{\alpha_{\delta,j}}(i) <_J g_\beta(i) \} \in D.$
\end{center}
Since $D$ is uniform, $[j, \kappa)\in D,$ so $u\cap [j, \kappa)\in D.$ We show that
\begin{center}
$i\in u\cap [j, \kappa) \Rightarrow  f_\beta(i) <_J f_{\delta}(i) <_J g_{\delta}(i) <_J g_\beta(i).$
\end{center}
So let $i\in u\cap [j, \kappa).$ Then
\begin{itemize}
\item $f_\delta(i)=f_{\alpha_{\delta,i}}(i) \geq_J f_{\alpha_{\delta,j}}(i) >_J f_\beta(i),$
\item $f_\delta(i) <_J g_\delta(i),$ by $(a)$,
\item $g_\delta(i)=f_{\alpha_{\delta,i+1}}(i) <_J g_{\alpha_{\delta,i+1}}(i) <_J g_{\alpha_{\delta,j}}(i) <_J g_\beta(i).$
\end{itemize}
This completes our construction. We show that $\langle  \langle f_\alpha/D:\alpha<\lambda \rangle, \langle g_\alpha/D: \alpha<\lambda \rangle   \rangle$ witnesses a cut of $J^\kappa/D$. Suppose not. So there is $f_*\in J^\kappa$ such that for all $\alpha<\lambda, f_\alpha <_D f_* <_D g_\alpha.$ Let $\langle A_\xi : \xi<2^\kappa \rangle$ be an enumeration of $D$, and define $h:\lambda \rightarrow 2^\kappa$ by
\begin{center}
$h(\alpha)=\xi \Leftrightarrow A_\xi=\{i<\kappa:  f_\alpha(i) <_J f_*(i) <_J g_\alpha(i)\}.$
\end{center}
By our assumption there is $\delta\in S$ such that $h \upharpoonright a_\delta$ is constantly $\xi_\delta.$ This means that for all $j<\kappa, h(\alpha_{\delta,j})=\xi_\delta$, i.e.
\begin{center}
$A_{\xi_\delta}=\{ i<\kappa:  f_{\alpha_{\delta,j}}(i) <_J f_*(i) <_J g_{\alpha_{\delta,j}}(i)     \}.$
\end{center}
 Then for all $i\in A_{\xi_\delta},$
\begin{center}
$g_\delta(i)=f_{\alpha_{\delta,i+1}}(i) <_J f_*(i),$
\end{center}
and hence $g_\delta <_D f_*,$ a contradiction. It follows that  $(\lambda,\lambda)\in \mathscr{C}(D),$ and the theorem follows.
\end{proof}
In the next lemma we produce models in which the assumptions in Theorem 3.3 are satisfied.
\begin{lemma}
$(a)$ Assume $GCH$. Then there is a cardinal preserving generic extension of $V$ in which there are $\bar{C}, S$ and $\langle (a_\delta, \xi_\delta): \delta\in S  \rangle$ as above.

$(b)$ If we have a square  for $\lambda$, where $\lambda$ is a successor cardinal or a limit but not weakly compact cardinal, then we can manipulate to have $(c)-(d).$  In particular the above hypotheses are valid, if $V=L$ and $\lambda>\kappa^+$ is not weakly compact.
\end{lemma}
\begin{proof}
$(a)$ We force $\bar{C}, S$ and $\langle (a_\delta, \xi_\delta): \delta\in S  \rangle$  by initial segments. So let $\MPB$
be the set of all conditions of the form
\[
p= \langle  \tau, \bar{c}, s, \langle  (a_\delta, \xi_\delta): \delta \in s             \rangle             \rangle
\]
such that
\begin{enumerate}
\item $\tau < \lambda$,

\item $\bar{c}= \langle  c_\alpha: \alpha < \tau    \rangle$, where
\begin{enumerate}
\item  $c_\alpha\subseteq\alpha$, 

\item  $lim(\alpha) \Rightarrow sup(c_\alpha)=\alpha$ and  $otp(c_\alpha) = cf(\alpha)$,

\item  $\beta\in c_\alpha \Rightarrow c_\beta=c_\alpha\cap\beta,$

\item  If $\alpha$ is a successor ordinal, then either $c_\alpha$ is empty, or has a last element.
\end{enumerate}
\item $s=\{\delta<\tau: otp(c_\delta)=\kappa, \delta\notin \bigcup_{\alpha<\tau}c_\alpha \}.$

\item
The sequence $\langle (a_\delta, \xi_\delta): \delta\in s  \rangle$ satisfies:
\begin{enumerate}
\item    $a_\delta\subseteq nacc(c_\delta)$,

\item  $otp(a_\delta)=\kappa,$

\item  $\xi_\delta<2^\kappa.$

\end{enumerate}
\end{enumerate}
Given $p \in \MPB,$ we denote it by $p= \langle  \tau^p, \bar{c}^p, s^p, \langle  (a^p_\delta, \xi^p_\delta): \delta \in s^p             \rangle             \rangle$.
For $p, q \in \MPB,$ the order relation $p \leq q$ is defined in the natural way, i.e., we require
\begin{enumerate}
\item $\tau^p \geq \tau^q$,

\item $\bar{c}^q=\bar{c}^p \upharpoonright \tau^q$,

\item $s^q = s^p \cap \tau^q$,

\item $\langle  (a^q_\delta, \xi^q_\delta): \delta \in s^q             \rangle=\langle  (a^p_\delta, \xi^p_\delta): \delta \in s^q             \rangle$.

\end{enumerate}
The forcing notion $\MPB$ is easily seen to be $\lambda$-distributive and $\lambda^+$-c.c., hence it preserves all cardinals and cofinalities.
Let $G$ be $\MPB$-generic over $V$, and define $\bar{C}=\langle C_\alpha: \alpha < \lambda \rangle, S$ and 
$\langle (a_\delta, \xi_\delta): \delta \in S  \rangle$
by

$\hspace{1.5cm}$ $C_\alpha=c^p_\alpha,$ for some (and hence all) $p \in G$ with  $\tau^p > \alpha$,

$\hspace{1.5cm}$ $S=\bigcup_{p \in G} s^p$,

$\hspace{1.5cm}$ $(a_\delta, \xi_\delta)=(a^p_\delta, \xi^p_\delta)$, for some (and hence all) $p \in G$ with  $\tau^p > \alpha$,

We show that $\bar{C}=\langle C_\alpha: \alpha < \lambda \rangle, S$ and
$\langle (a_\delta, \xi_\delta): \delta \in S  \rangle$ are as required. It suffices to show that $S$ is stationary and that if 
$h: \lambda\rightarrow 2^\kappa,$ then there is some $\delta\in S$ such that $h\upharpoonright a_\delta$ is constantly $\xi_\delta.$
\\
{\bf $S$ is a stationary subset of $\lambda$:} Assume $p \in \MPB$ and $p \Vdash$``$\lusim{D}$ is a club subset of $\lambda$''.
Let $\theta > \lambda$ be large enough regular and let $\unlhd$
be a well-ordering of $H(\theta)$. Assume $M \prec \langle  H(\theta), \in, \unlhd     \rangle$
is such that $M$ contains all relevant information, in particular $\MPB, p, \lusim{C}, \dots, \in M$, $|M| < \lambda, ^{< \kappa}M \subseteq M$
and $\delta= M \cap \lambda$ is an ordinal with $cf(\delta)=\kappa$.

By recursion, we can define a decreasing sequence $\langle  p_\alpha: \alpha < \kappa       \rangle$
of conditions in $\MPB$ such that
\begin{itemize}
\item $p_0=p,$

\item For each $\alpha < \kappa, p_\alpha \in M,$ 

\item $\langle \tau^{p_\alpha}: \alpha < \kappa  \rangle$
is a normal  sequence cofinal in $\delta$,

\item $p_{\alpha+1} \Vdash$``$\lusim{C} \cap (\tau^{p_\alpha}, \tau^{p_{\alpha+1}}) \neq \emptyset$'',

\item If we define $q$ by $\tau^q=\delta+2$ so that 
\begin{itemize}
\item $q \upharpoonright \delta= \bigwedge_{\alpha < \kappa}p_\alpha$ (the greatest lower bound of $p_\alpha$'s, $\alpha < \kappa$),
\item  $c^q_\delta= \bigcup_{\xi < \kappa} c_\xi$, where $c_\xi=c_\xi^{p_\alpha}$ for some (and hence any) $\alpha$ with $\xi < \tau^{p_\alpha}$, 
\item $c^q_{\delta+1}=\{\delta\}$,
\end{itemize}
then $q \in \MPB$.
\end{itemize}
Then $q \Vdash$``$\delta \in \lusim{C} \cap \lusim{S}$,
and we are done.
\\
{\bf Clause $(e)$-$(\delta)$ of Theorem 3.3 holds:} suppose $h: \lambda\rightarrow 2^\kappa$ and let $\lusim{h}$
be a name for it. Also let $p \in \MPB$ forces ``$\lusim{h}: \lambda\rightarrow 2^\kappa$ is a function''. 
Define a decreasing chain $\langle p_\alpha: \alpha < \lambda     \rangle$ of conditions in $\MPB$ such that
\begin{itemize}
\item $p_0=p,$
\item $p_{\alpha+1} \|$``$\lusim{h} \upharpoonright \tau^{p_\alpha}$'', say $p_{\alpha+1} \Vdash$``$\lusim{h} \upharpoonright \tau^{p_\alpha}=a_\alpha$'',
\item For limit ordinal $\alpha, p_\alpha = \bigwedge_{\beta<\alpha}p_\beta.$
\end{itemize}
As $\lambda > 2^\kappa,$ we can find $\alpha < \lambda, \xi < 2^\kappa$
and $b \subseteq a_\alpha$ of order-type $\kappa$ such that $p_{\alpha+1}\Vdash$``$\lusim{h} \upharpoonright b = id_\xi \upharpoonright b$'', where $id_\xi$ is the constant function taking value $\xi$ everywhere.

As above, we can extend $p_{\alpha+1}$ to some condition $q$ such that $\tau^q=\delta+2$ for some $\delta \in S$
such that $(a^q_\delta, \xi^q_\delta)=(b, \xi)$ and  $b \subseteq nacc(c^q_\delta)$. Then $q \Vdash$``$\delta \in \lusim{S}$ and $\lusim{h} \upharpoonright a_\delta$ is constantly $\xi_\delta$, where $a_\delta \subseteq nacc(c_\delta)$ has order type $\kappa$''.
We are done.

$(b)$ As $\lambda$ has a square, we can find a
set $T\subseteq\lambda$ and a sequence $\bar{D}=\langle D_\alpha: \alpha\in T \rangle$ such that:
\begin{itemize}
\item $\lambda \setminus T$ in a non-stationary subset of $\lambda,$

\item $ D_\alpha$ is a club of $\alpha,$

\item $\beta\in D_\alpha \Rightarrow \beta\in T$ and $D_\beta=D_\alpha\cap\beta,$

\item $otp(D_\a) <\a.$
\end{itemize}
We assume $T$ contains only limit ordinals and for $\alpha\notin T$ we set $D_\alpha=\emptyset.$  As $\lambda \setminus T$ is non-stationary, we can find a club $D \subseteq T.$
We now use the sequence $\bar{D}$ to define  a new sequence $\bar{C}=\langle C_\alpha: \alpha < \lambda \rangle$ 
such that $\bar{C}$ satisfies  clause $(c)$ of Theorem 3.3 and further
\[
\delta \in D ~~\& ~~ cf(\delta) =\kappa \Rightarrow \delta\notin \bigcup_{\alpha<\lambda}C_\alpha.
\]
It is clear that 
$S=\{\delta<\lambda: otp(C_\delta)=\kappa, \delta\notin \bigcup_{\alpha<\lambda}C_\alpha \} \supseteq D \cap Cf(\kappa)$, hence it is stationary.
The second part follows from results of Jensen. We are done.
\end{proof}

\begin{corollary}
$(V=L)$ Suppose $D$ is a uniform ultrafilter on some cardinal $\k.$ Then $\mathscr{C}(D)$ is a proper class.
\end{corollary}
\begin{proof}
Assume $V=L.$ Let $\lambda > 2^\kappa$ be a successor cardinal. By Lemma 3.4$(b)$, we can find a sequence $\bar{C} = \langle C_\alpha: \alpha < \lambda \rangle$
which satisfies clause $(c)$ of Theorem 3.3 and such that the set $S=\{\delta<\lambda: otp(C_\delta)=\kappa, \delta\notin \bigcup_{\alpha<\lambda}C_\alpha \}$
is stationary in $\lambda$. Also we can apply $\Diamond_S$ to find a sequence $\langle (a_\delta, \xi_\delta): \delta \in S       \rangle$
satisfying clause $(e)$ of Theorem 3.3. It follows from Theorem 3.3 that $(\lambda, \lambda) \in \mathscr{C}(D).$ Thus
\[
\mathscr{C}(D) \supseteq \{ (\lambda, \lambda): \lambda > 2^\kappa \text{~is a successor cardinal~}              \},
\]
and hence $\mathscr{C}(D)$ is a proper class.
\end{proof}

\section{More on $\mathscr{C}(D)$ being small: consistency results}
\begin{lemma}
Suppose that:

$\hspace{1.cm}$$(a)$ $\sigma, \k, \mu, \l$ and $\tau$ are such that:

$\hspace{1.5cm}$$(\alpha)$ $\sigma< \k=cf(\k)\leq \mu < \l=cf(\l)$ are infinite cardinals,

$\hspace{1.5cm}$$(\beta)$  $\tau:\l\times\l \rightarrow \sigma,$

$\hspace{1.5cm}$$(\gamma)$ $\a<\l \Rightarrow |\a|^{<\k}<\l,$

$\hspace{1.cm}$$(b)$ We have $\mathcal{U},\tau_1, \bar{\a}$  such that:

$\hspace{1.5cm}$$(\alpha)$ $\mathcal{U}\subseteq S^\l_{\geq \k}$ is stationary, where $S^\l_{\geq \k}=\{\a<\l: cf(\a)\geq \k   \},$

$\hspace{1.5cm}$$(\beta)$ $\tau_1: \mathcal{U} \rightarrow \sigma,$

$\hspace{1.5cm}$$(\gamma)$ $\bar{\a}= \langle \beta_{v,\xi}: v\in [\mathcal{U}]^{<\k}, \xi<\mu  \rangle$ is a sequence of ordinals $<\l,$

$\hspace{1.5cm}$$(\delta)$ $\beta_{v,\xi} > sup(v),$

$\hspace{1.5cm}$$(\epsilon)$ If $v\in [\mathcal{U}]^{<\k},$ then for some $\xi<\mu,$ we have $\sup(v)<\xi$ and
\begin{center}
$\a\in v \Rightarrow \tau(\a, \beta_{v,\xi})=\tau_1(\a).$
\end{center}
 Then:

 $\hspace{1.7cm}$ There are a club $E$ of $\l$ and $\tau_2: \mathcal{U}\cap E \rightarrow \sigma$ such that if  $v\in [\mathcal{U}]^{<\k}$

 $(c)$ $\hspace{1.1cm}$ and $sup(v) <\delta\in \mathcal{U}\cap E,$ then for some $\beta\in (sup(v), \delta)$ we have

  $\hspace{1.7cm}$ $\a\in v \Rightarrow \tau(\a,\b)=\tau_1(\a)$ and $\tau(\b, \delta)=\tau_2(\delta).$

\end{lemma}
\begin{proof}
For every $v\in [\mathcal{U}]^{<\k}$ set
\begin{center}
$J_v=\{w\subseteq \mu:$ there is no $i\in w$ such that $(\forall \a\in v) \tau(\a, \b_{v,i})=\tau_1(\a) \}.$
\end{center}
$J_v$ is clearly a $\k-$complete ideal on $\mu$ and $\mu\notin J_v$ by $(b)(\epsilon)$ (and in fact $J_v$ has a maximal element). Let

$\hspace{1.cm}$ $E=\{\delta<\l: \delta$ is a limit ordinal, and for every bounded  subset $v$  of

$\hspace{3.5cm}$ $\delta$ of cardinality less than $\k,$ we have $\bigcup \{\beta_{v,i}: i<\mu \} \subseteq \delta  \}.$

Clearly $E$ is a club of $\l.$
\begin{claim}
For every $\delta\in S^\l_{\geq\k}\cap E,$ there is an ordinal $\xi<\sigma$  such that if $v_1\in [\mathcal{U}\cap \delta]^{<\k},$ then for some $v$ we have:

$\hspace{1.cm}$$(a)$ $v_1\subseteq v\in [\mathcal{U}\cap \delta]^{<\k},$

$\hspace{1.cm}$$(b)$ $\{i<\mu: \tau(\b_{v,i},\delta)=\xi \}\in J_v^+.$
\end{claim}
\begin{proof}
Suppose not. Then for each $\xi<\sigma,$ we can find some $v_{\xi}\in [\mathcal{U}\cap \delta]^{<\k}$ such that if $v_\xi \subseteq v\in [\mathcal{U}\cap \delta]^{<\k},$ then $w_{\xi,v}=\{i<\mu: \tau(\beta_{v,i}, \delta)=\xi \}\in J_v.$ Let $v=\bigcup_{\xi<\sigma}v_\xi.$ Then $v\in [\mathcal{U}\cap \delta]^{<\k}$ and for all $\xi<\sigma, w_{\xi,v}\in J_v,$ so by $\k-$completeness of $J_v$, $w=\bigcup_{\xi<\sigma}w_{\xi,v}\in J_v.$ Clearly $w=\mu,$ so $\mu\in J_v,$ which contradicts $(b)(\epsilon).$
\end{proof}

For $\delta \in \mathcal{U}\cap E$ let  $\tau_2(\delta)$ be the least  $\xi$ as in Claim 4.2. We show that $E$ and $\tau_2$ are as required. So let $v\in [\mathcal{U}]^{<\k}$ and suppose that $sup(v)<\delta\in E\cap \mathcal{U}.$ By Claim 4.2, there is $u$ such that $v\subseteq u\in [\mathcal{U}\cap \delta]^{<\k}$ and $w=\{i<\mu: \tau(\b_{u,i},\delta)=\tau_2(\delta) \}\in J_u^+.$ Since $w\notin J_u,$ there exists $i\in w$ such that for all $\a\in u, \tau(\a, \beta_{u,i})=\tau_1(\a).$ So it suffices to take $\b=\b_{u,i}$.
\end{proof}
Before we continue, let us recall from Theorem 2.22  the principle $(*)^{2,2}_{\l,\sigma}$, which says  if $c:[\lambda]^2 \rightarrow \sigma,$ then there are $u\subseteq \sigma, |u|< 3$ and $S\in [\lambda]^\lambda$ such that
 if $\alpha<\beta$ are in $S$, then for some $n< 3$ and $\gamma_0< \dots <\gamma_n$ we
 have $\gamma_0=\alpha, \gamma_n=\beta$ and for $l<n, c\{\gamma_l,\gamma_{l+1}\}\in u.$

\begin{lemma}
Suppose that $\sigma, \k, \mu$ and $\lambda$ are infinite cardinals such that $\sigma<\k=cf(\k)\leq \mu <\l=cf(\l)$ and suppose that for each $\tau:\l\times\l \rightarrow \sigma$ there are $\mathcal{U}, \tau_1$ and $\bar{\alpha}$ such that $(a)$ and $(b)$ of Lemma 4.1 hold. Then $(*)^{2,2}_{\l,\sigma}$ holds.
\end{lemma}
\begin{proof}
Fix $\tau,$ and let  $\mathcal{U}, \tau_1$ and $\bar{\alpha}$ witness  $(a)$ and $(b)$ of Lemma 4.1 hold.
Let $E$ and $\tau_2$ be as in the conclusion of Lemma 4.1. Since $ran(\tau_1) \subseteq \sigma <\l$, for some $\xi_1<\sigma,$ the set
\begin{center}
$S_1=\{\a\in \mathcal{U}\cap E: \tau_1(\a)=\xi_1 \}$
\end{center}
is a stationary subset $\l.$ So again as $ran(\tau_2) \subseteq \sigma <\l$, for some $\xi_2<\sigma,$ the set
\begin{center}
$S_2=\{\a\in S_1: \tau_2(\a)=\xi_2 \}$
\end{center}
is a stationary subset of $S_1.$ Thus
 $S_2$ is of size $\l.$ We show that $u=\{\xi_1, \xi_2 \}$ and $S_2$ witness $(*)^{2,2}_{\l,\sigma}$. So suppose that $\a<\b$ are in $S_2$. By $(c)$ of Lemma 4.1 (taking $v=\{\a\}$ and $\delta=\b$) we can find some $\gamma \in (\a,\b)$ such that $\tau(\a,\gamma)=\tau_1(\a)=\xi_1$ and $\tau(\gamma,\b)=\tau_2(\b)=\xi_2.$ Thus $\tau(\a,\gamma), \tau(\gamma,\b)\in u,$  as required.
\end{proof}

\begin{theorem}
Suppose that:

$\hspace{1.cm}$$(\a)$ $\rho < \k \leq \chi=cf(\chi)$ and $\chi=\chi^{<\k},$

$\hspace{1.cm}$$(\b)$ $\mu >\chi$ is a supercompact cardinal,

$\hspace{1.cm}$$(\gamma)$ $\mathbb{Q}=Add(\chi, \mu)$ is  the Cohen forcing  for adding $\mu-$many new Cohen

$\hspace{1.3cm}$ subsets of $\chi$,

$\hspace{1.cm}$$(\delta)$ $\l=cf(\l) >\mu$ and $\a<\l \Rightarrow |\a|^{<\chi} <\l.$

Then the following hold in $V^{\mathbb{Q}}$:

$\hspace{1.cm}$$(a)$ If $\tau: \l\times\l \rightarrow \rho,$ then for some $\mathcal{U},\tau_1$ and $\bar{\a}$ clauses $(a)$ and $(b)$ of Lemma 4.1

$\hspace{1.3cm}$ hold,

$\hspace{1.cm}$$(b)$ $(*)^{2,2}_{\l,\rho}.$
\end{theorem}
\begin{proof}
Note that  $(b)$ follows from $(a)$ and Lemma 4.3, so it suffices to prove $(a)$. Let $\lusim{\tau}$ be a $\mathbb{Q}-$name for $\tau,$ and suppose for simplicity that $1_{\mathbb{Q}}\Vdash \lusim{\tau}: \l\times\l \rightarrow \rho.$ Let $D$ be a normal measure on $I=P_\mu(\l),$ and let $j: V \rightarrow M \simeq Ult(V, D)$ be the corresponding ultrapower embedding so that $\mu=crit(j), j(\mu)>\l$ and $M^{\l} \subseteq M.$ The following claim is trivial using the fact that $\mathbb{Q}$ satisfies the $\sigma^+-c.c.,$ where $\sigma=2^{<\chi}.$
\begin{claim}
For any $\a<\b<\l,$ there are $(\bar{q}_{\a,\b}, \bar{c}_{\a,\b})$ such that:

$\hspace{1.cm}$$(a)$ $\bar{q}_{\a,\b}= \langle q_{\a,\b,\xi}: \xi<\sigma \rangle$ is a maximal antichain of $\mathbb{Q}$,

$\hspace{1.cm}$$(b)$ $\bar{c}_{\a,\b}= \langle c_{\a,\b,\xi}: \xi<\sigma  \rangle,$ where $c_{\a,\b,\xi}<\k$

$\hspace{1.cm}$$(c)$ $q_{\a,\b,\xi} \Vdash \lusim{\tau}(\a,\b)=c_{\a,\b, \xi}.$
\end{claim}

For $u\in I$  set
$\b(u)=\b_u=sup(u)<\l.$ By Solovay \cite{solovay},  there exists $A\in D$ such that for all $u, v\in A,$ if $\b(u)=\b(v),$ then $u=v.$
Let $A_1=\{u\in A: \partial_u=otp(u)$ is a regular cardinal  such that $\a< \partial_u \Rightarrow |\a|^{<\chi} < \partial_u$ and $\delta=sup(u\cap\delta)$ $\&$ $ cf(\delta)=\k^+ \Rightarrow \delta\in u  \}.$
\begin{claim}
$A_1\in D$.
\end{claim}
\begin{proof}
By $(\delta),$ $\l=otp(j[\l])$ is a regular cardinal, and for all $\a<\l, |\a|^{<\chi}<\l.$ Now suppose that $\delta=sup(j[\l]\cap\delta)$ and $ cf(\delta)=\k^+$. Then clearly $\delta=j(\a)$ for some $\a<\l,$ and hence $\delta\in j[\l].$ It follows that $j[\l]\in j(A_1),$ and hence $A_1\in D.$
\end{proof}
For $u\in A_1$ set $v_u=\{\delta < sup(u): cf(\delta)=\k^+$ and $\delta=sup(u\cap \delta)\}.$ Then by definiton of $A_1$,
 $v_u$ is an unbounded subset of $u$,
 $v_u \subseteq \{\delta\in u: cf(\delta)=\k^+  \}$ and $v_u$ is a stationary subset of $sup(u).$
\begin{claim}
There are $A_2$ and  $\mathcal{U}$ such that:

$\hspace{1.cm}$$(a)$ $A_2\in D$ and $A_2\subseteq A_1,$

$\hspace{1.cm}$$(b)$ $\mathcal{U}\subseteq \l$ and $u\in A_2 \Rightarrow v_u=\mathcal{U}\cap u,$

$\hspace{1.cm}$$(c)$ $\mathcal{U} \subseteq \{\delta<\l: cf(\delta)=\k^+ \}$ is a stationary subset of $\l.$
\end{claim}
\begin{proof}
Set
\begin{center}
$\mathcal{U}=\{\delta<\l: cf(\delta)=\k^+$ and $sup(j[\l]\cap j(\delta))=j(\delta)\}.$
\end{center}
$\mathcal{U}$ is clearly a stationary subset of $\l$, consisting of ordinals of cofinality $\k^+.$ Let
\begin{center}
$A_2=\{u\in A_1: v_u=\mathcal{U}\cap u  \}.$
\end{center}
It suffices to show that $A_2\in D.$ We have

$\hspace{1.7cm}$$v_{j[\l]}=\{\delta < sup(j[\l]): cf(\delta)=\k^+$ and $sup(j[\l] \cap\delta)=\delta  \}$

$\hspace{2.5cm}$$=\{j(\delta): \delta<\l, cf(\delta)=\k^+$ and $sup(j[\l] \cap j(\delta))=j(\delta)  \}$

$\hspace{2.5cm}$$=j[\mathcal{U}]$

$\hspace{2.5cm}$$=j(\mathcal{U})\cap j[\l].$

This implies $j[\l]\in j(A_2),$ or equivalently $A_2\in D$, as required.
\end{proof}
Fix $\a<\l$ and $\xi<\sigma.$ By $\mu-$completeness of $D$ and $\k<\mu,$ we can find some $c_{\a,\xi}<\k$ such that
\begin{center}
$A_2(\a,\xi)=\{u\in A_2: \a\in u$ and $c_{\a, \b(u), \xi}=c_{\a,\xi}  \}\in D.$
\end{center}
Let $A(\a, \xi)=\{u\in A_2(\a,\xi): u\cap \mu\in \mu$ and $ dom(q_{\a, \b(u), \xi})\cap u \prec u \},$ where $\prec$ is the Magidor relation (see \cite{magidor}) on $I$ defined by $u \prec w$ iff $u \subseteq w$ and $otp(u) < w\cap\mu$.  Since the forcing conditions have size $<\chi$ and $\chi<\mu,$ we can easily conclude that $A(\a, \xi)\in D.$

Define $F: A(\a, \xi) \rightarrow I$ by $F(u)=q_{\a, \b(u), \xi} \upharpoonright u.$ As $D$ is normal, it follows from \cite{magidor} that there exist $B(\a, \xi)$ and $q_{\a,\xi}$ such that:
\begin{itemize}
\item $B(\a,\xi) \subseteq A(\a,\xi),$
\item $B(\a,\xi)\in D,$
\item $q_{\a,\xi}\in \mathbb{Q},$
\item For all $u\in B(\a,\xi), q_{\a, \b(u), \xi} \upharpoonright u =q_{\a,\xi}.$
\end{itemize}

So, varying $\xi,$ it follows from the $\mu-$completeness of $D$ that
\begin{center}
$B(\a)= \bigcap_{\xi<\chi}B(\a, \xi)\in D.$
\end{center}
Set
\begin{center}
$B= \{u\in I: \a\in u \Rightarrow u\in B(\a)  \}.$
\end{center}
By normality of $D$, $B\in D.$

Now for each $v\in [\mathcal{U}]^{< \k}$ we choose $u_{v,\xi}$ by induction on $\xi<\sigma$ such that:
\begin{itemize}
\item $u_{v,\xi}\in B,$

\item $v \subseteq u_{v,\xi},$
\item $\xi <\zeta \Rightarrow u_{v,\xi} \subseteq u_{v, \zeta}$ and $\bigcup \{dom(q_{\a, \b(u_{v,\xi}), \epsilon}): \a\in u_{v,\xi}, \epsilon<\sigma   \}\subseteq u_{v, \zeta}.$
\end{itemize}

For $v\in [\mathcal{U}]^{<\k}$ and $\xi<\sigma$ let
\begin{center}
$\beta_{v,\xi}=\beta(u_{v,\xi}),$
\end{center}
and define
\begin{center}
$\bar{\a}= \langle \beta_{v,\xi}: v\in [\mathcal{U}]^{<\k}, \xi<\sigma  \rangle$.
\end{center}
Let $G$ be $\mathbb{Q}-$generic over $V$.
Define $\tau_1: \mathcal{U} \rightarrow \k, \tau_1\in V[G],$ by
\begin{center}
$\tau_1(\a)=c_{\a,\xi},$ where $\xi<\sigma$ is the least ordinal such that $q_{\a,\xi}\in G.$
\end{center}
\begin{claim}
$\tau_1(\a)$ is well-defined.
\end{claim}
\begin{proof}
Let $u\in B$ be such that $\a\in u.$ Then $\langle q_{\a, \b(u), \xi}: \xi<\sigma \rangle$ is a maximal antichain of $\mathbb{Q},$ so for some unique  $\xi_u <\sigma, q_{\a, \b(u),\xi_u}\in G.$ Since $q_{\a, \b(u),\xi_u}\leq q_{\a,\xi_u},$ we have $q_{\a,\xi_u}\in G$. It follows that $\{\xi<\sigma: q_{\a,\xi}\in G \}\neq \emptyset,$ and hence $\tau_1(\a)$ is well-defined.
\end{proof}
The following claim completes the proof of the theorem.
\begin{claim}
In $V[G]$, $\mathcal{U}, \tau, \tau_1$ and $\bar{\a}$ are as required in $(a)$ and $(b)$ of Lemma 4.1.
\end{claim}
\begin{proof}
It suffices to prove $(b)(\epsilon).$ So let $v\in [\mathcal{U}]^{<\k}.$ Clearly $v\in V$. Let $q\in G.$ We find $q^*\leq q$ and $i<\sigma$ such that
\begin{center}
$q^*\Vdash$ $\a\in v \Rightarrow \lusim{\tau}(\a, \b_{v,i})=\lusim{\tau_1}(\a),$
\end{center}
where $\lusim{\tau}, \lusim{\tau_1}$ are $\mathbb{Q}-$names for $\tau, \tau_1$ respectively. Let $q_0=q.$
As the forcing is $\chi-$closed, and $\chi \geq \k,$ there exist $q_1 \leq q_0$ and $\langle \xi(\a): \a\in v   \rangle\in V$ such that for each $\a\in v,$
\begin{center}
$q_1 \Vdash \lusim{\xi}_\a= min\{\xi<\sigma:  q_{\a, \xi}\in \dot{G}\}=\xi(\a),$
\end{center}
where $\dot{G}$ is the canonical $\mathbb{Q}-$name for the generic filter $G$. Again as  $\mathbb{Q}$ is $\chi-$closed and $\chi \geq \k,$ we can find $q_2 \leq q_1$  such that for each $\a\in v$
\begin{center}
$q_2 \Vdash q_{\a, \xi(\a)} \in \dot{G}.$
\end{center}
Since the forcing is separative, $q_2 \leq q_{\a, \xi(\a)}$ for all $\a\in v.$ Since $|dom(q_2)| <\chi<\mu,$  we can find $i<\sigma$ such that
\begin{center}
$\bigcup \{dom(q_{\a, \b_{v,i},\xi}): \a\in v, \xi<\sigma \}$
\end{center}
is disjoint from $dom(q_2)\setminus \bigcup \{dom(q_{\a, \xi(\a)}): \a\in v \}.$ Let
\begin{center}
$q^*=q_2 \cup \bigcup \{q_{\a, \b_{v,i}, \xi(\a)}: \a\in v \}.$
\end{center}
If $q^*$ is not a well-defined function,  then we can find some $\sigma<\mu$ and $\a\in v$ such that both of $q_2(\sigma)$ and $q_{\a, \b_{v,i}, \xi(\a)}(\sigma)$ are defined and are not equal. By our choice of $i$, $q_2(\sigma)=q_{\b, \xi(\b)}(\sigma),$ for some $\b\in v.$ But then $q_{\a, \xi(\a)}$ and $q_{\b, \xi(\b)}$ are incompatible, which is in contradiction with $q_2 \leq q_{\a,\xi(\a)}, q_{\b, \xi(\b)}.$ So $q^*$ is well-defined.

It follows that  $q^*\in\mathbb{Q}.$ Let $\a\in v.$ We show that
\begin{center}
$q^*\Vdash$ $\lusim{\tau}(\a, \b_{v,i})=\lusim{\tau_1}(\a).$
\end{center}
We have
\begin{center}
$q_{\a, \b_{v,i}, \xi(\a)}\Vdash \lusim{\tau}(\a, \b_{v, i})=c_{\a, \b_{v, i}, \xi(\a)}=c_{\a, \b(u_{v,i}), \xi(\a)}=c_{\a, \xi(\a)}.$
\end{center}
On the other hand
\begin{center}
$q_2 \Vdash \lusim{\tau_1}(\a)=c_{\a, \xi(\a)}.$
\end{center}
So as $q^* \leq q_2, q_{\a, \b_{v,i}, \xi(\a)},$
\begin{center}
$q^* \Vdash \lusim{\tau}(\a, \b_{v, i})=c_{\a, \xi(\a)}=\lusim{\tau_1}(\a).$
\end{center}
The claim follows.
\end{proof}
This completes the proof of Theorem 4.4.
\end{proof}
\begin{corollary}
Assume $\mu$ is a supercompact cardinal and $\theta=(2^\k)^+ <\mu.$ Then for some $(2^{<\theta})^+-c.c., \theta-$closed forcing notion $\mathbb{Q}$ of size $\mu,$ the following hold in $V^{\mathbb{Q}}$:

$\hspace{1.cm}$$(a)$ $2^\theta=\mu,$

$\hspace{1.cm}$$(b)$ If $\l=cf(\l)\geq \mu$ and $(\forall \a<\l) |\a|^{<\theta} <\l,$ then $(*)^{2,2}_{\l,\theta}$.

$\hspace{1.cm}$$(c)$ If $D$ is an ultrafilter on $\k, \l=cf(\l)\geq \mu$ and $(\forall \a<\l) |\a|^{<\theta} <\l,$ then

$\hspace{1.6cm}$$(\l, \l)\notin \mathscr{C}(D),$
\end{corollary}
\begin{proof}
Let $\mathbb{Q}=Add(\theta,\mu).$ Then $(a)$ is clear, $(b)$ follows from Theorem 4.4, and $(c)$ follows from $(b)$ and Theorem 2.22.
\end{proof}
We now give an application of Theorem 4.4 in depth and depth$^+$ of Boolean algebras, which continues the works in \cite{garti-shelah1}, \cite{garti-shelah2}, \cite{garti-shelah3}, \cite{garti-shelah4} and \cite{shelah4}. Recall that if $\mathbb{B}$ is a Boolean algebra, then its depth and depth$^+$ are defined as follows:

$\hspace{1.cm}$ $Depth(\mathbb{B})=sup\{\theta:$ there exists a $\theta-$increasing sequence in $\mathbb{B}\}.$

$\hspace{1.cm}$ $Depth^+(\mathbb{B})=sup\{\theta^+:$ there exists a $\theta-$increasing sequence in $\mathbb{B}\}.$

\begin{corollary}
Assume $\mu$ is a supercompact cardinal, $\theta=(2^\k)^+ <\mu$ and $\l=cf(\l)\geq \mu$. Then for some $(2^{<\theta})^+-c.c., \theta-$closed forcing notion $\mathbb{Q}$ of size $\mu,$ the following holds in $V^{\mathbb{Q}}$: If $D$ is an ultrafilter on $\k$, $\langle \mathbb{B}_i:i<\k  \rangle$ is a sequence of Boolean algebras satisfying $i<\k \Rightarrow Depth^+(\mathbb{B}_i)\leq \l$ and $\mathbb{B}=\prod_{i<\k}\mathbb{B}_i /D,$ then $Depth^+(\mathbb{B})\leq \l.$
\end{corollary}
\begin{proof}
Let $\mathbb{Q}=Add(\theta,\mu),$ and let $G$ be $\mathbb{Q}-$generic over $V$. We work in $V[G].$ Suppose that $D$ is an ultrafilter on $\k$, $\langle \mathbb{B}_i:i<\k  \rangle$ is a sequence of Boolean algebras such that for $i<\k$ $Depth^+(\mathbb{B}_i)\leq \l$ and let $\mathbb{B}=\prod_{i<\k}\mathbb{B}_i /D.$ We show that $Depth^+(\mathbb{B})\leq \l.$

Suppose not. So we can find an increasing sequence $\langle a_\a: \a<\l \rangle$ of elements of $\mathbb{B}.$ Let us write $a_\a = \langle a^\a_i: i<\k \rangle / D$ for every $\a<\l.$ For $\a<\b<\l, a_\a <_{\mathbb{B}} a_\b,$ and hence
\begin{center}
$A_{\a,\b}=\{i<\k: a^\a_i <_{\mathbb{B}_i} a^\b_i  \}\in D.$
\end{center}
Define $c: [\l]^2  \rightarrow D$ by $c(\a,\b)=A_{\a,\b}.$ By Theorem 4.4, $(*)^{2,2}_{\l,2^\k}$ holds in $V[G]$, so we can find an unbounded subset $S$ of $\l,$ and $A_0, A_1\in D$ such that if $\a<\b$ are in $S$, then for some $\a <\gamma <\b,$ we have $A_{\a,\gamma}=A_0$ and $A_{\gamma,\b}=A_1.$ Let $A=A_0\cap A_1,$ and fix some $i_*\in A.$ Then for all $\a<\b$ in S, $a^\a_{i_*} <_{\mathbb{B}_{i_*}} a^\b_{i_*}.$ So $\langle a^\a_{i_*}: \a\in S  \rangle$ is an increasing sequence in $\mathbb{B}_{i_*},$ hence $Depth^+(\mathbb{B}_{i_*})\geq \l^+,$ which contradicts the assumption $Depth^+(\mathbb{B}_{i_*})\leq \l$.
\end{proof}
The next theorem can be proved as in  Theorem 4.4 and corollaries 4.10 and 4.11.
\begin{theorem}
Suppose that:

$\hspace{1.cm}$$(\a)$ $\k \leq \chi=cf(\chi)$ and $\chi=\chi^{<\k},$

$\hspace{1.cm}$$(\b)$ $\mu >\chi$ is a weakly compact cardinal,

$\hspace{1.cm}$$(\gamma)$ $\mathbb{Q}=Add(\chi, \mu)$ is  the Cohen forcing  for adding $\mu-$many new Cohen

$\hspace{1.3cm}$ subsets of $\chi$,

Then in $V^{\mathbb{Q}}$ the following hold:

$\hspace{1.cm}$$(a)$ If $D$ is an ultrafilter on $\k,$ then $(\mu,\mu)\notin \mathscr{C}(D).$

$\hspace{1.cm}$$(b)$ $(*)^{2,2}_{\mu,\k}.$

$\hspace{1.cm}$$(c)$ If $D$ is an ultrafilter on $\k$, $\langle \mathbb{B}_i:i<\k  \rangle$ is a sequence of Boolean algebras

$\hspace{1.5cm}$satisfying
$i<\k \Rightarrow Depth^+(\mathbb{B}_i)\leq \mu$ and $\mathbb{B}=\prod_{i<\k}\mathbb{B}_i /D,$ then $Depth^+(\mathbb{B})\leq \mu.$
\end{theorem}

\section{a global consistency result}

In this section we prove the consistency of `` if $(\l_1, \l_2)\in \mathscr{C}(D),$ where $D$ is an ultrafilter on $\k,$ then $\l_1 + \l_2 < 2^{2^\k}$''.

\begin{lemma}
Suppose that:

$\hspace{1.cm}$$(\a)$ $\k <\theta=cf(\theta),$

$\hspace{1.cm}$$(\b)$ $\l_1, \l_2$ are regular cardinals and $\l_1 + \l_2 > 2^{<\theta},$

$\hspace{1.cm}$$(\gamma)$ $\mathbb{Q}_l= Add(\theta, \mathcal{U}_l), l=1,2,$ where $\mathcal{U}_1 \subseteq \mathcal{U}_2$ are two sets of ordinals (hence $\mathbb{Q}_1 \lessdot \mathbb{Q}_2$),

$\hspace{1.cm}$$(\delta)$  $\lusim{\bar{f}}^l = \langle \lusim{f}^l_\a: \a <\l_l \rangle, l=1,2$ where $\lusim{f}^l_\a$ is a $\mathbb{Q}_1-$name for a function from $\k$  to

$\hspace{1.5cm}$$\lusim{I}.$

$\hspace{1.cm}$$(\epsilon)$ $\Vdash_{\mathbb{Q}_1} ``\maltese$'', where

$\hspace{2.5cm}$ $\lusim{I}$ is a linear order, $\lusim{D}$ is an ultrafilter on $\k$ and $(\lusim{\bar{f}}^1/D, \lusim{\bar{f}}^2/D)$

$(\maltese):$ $\hspace{1.6cm}$ represents a $(\l_1, \l_2)$-pre-cut in $\lusim{I}^\k / \lusim{D}$ which is a pre-cut in

$\hspace{2.5cm}$  $\lusim{J}^\k / \lusim{D}$ for each linear order $\lusim{J} \supseteq \lusim{I}.$

Then  $\lusim{f}^l_\a, l=1,2$ and $\lusim{I}$ are also $\mathbb{Q}_2-$names, and $\Vdash_{\mathbb{Q}_2} `` \maltese$''.
\end{lemma}
\begin{proof}

Let $\mathbb{Q}=\mathbb{Q}_2 /\mathbb{Q}_1$. Then clearly  $\mathbb{Q}= Add(\theta, \mathcal{U}_2 \setminus \mathcal{U}_1),$ so we can assume without loss of generality that $\mathcal{U}_1=\emptyset,$ so that $\mathbb{Q}_1$ is the trivial forcing notion and $V=V^{\mathbb{Q}_1}$.

So $I, f^l_\a, D\in V$ are objects and not names. Also without loss of generality $\l_1 > 2^{<\theta}$, hence by Theorem 2.16, $\l_2 >2^{<\theta}$. Set $\mathcal{U}=\mathcal{U}_2= \mathcal{U}_2\setminus \mathcal{U}_1.$ We can also assume that the linear order $I$ mentioned in $`` \maltese$'' is a complete linear order whose set of elements is in $V$.

Toward contradiction assume that $\nVdash_{\mathbb{Q}_2} `` \maltese$.'' So we can find $q_*\in \mathbb{Q}_2$ and $\mathbb{Q}_2-$names $\lusim{J}$ and $\lusim{h}$ such that:

$\hspace{1.cm}$ $(\a)$ $q_* \in \mathbb{Q}_2,$

$\hspace{1.cm}$ $(\b)$ $q_* \Vdash `` \lusim{J}$ is a linear order such that $\lusim{J} \supseteq I$'',

$\hspace{1.cm}$ $(\gamma)$ $q_*\Vdash `` \lusim{h}\in$ $\lusim{J}^{\k}$'',

$\hspace{1.cm}$ $(\delta)$ $q_*\Vdash ``f^1_{\a_1} <_D \lusim{h} <_D f^2_{\a_2}$ for every $\a_1 < \l_1, \a_2 < \l_2$''.

{\bf Case 1. $(\forall \a <\l_1) |\a|^\k < \l_1$:} For each $\a <\l_1$ choose $(q_\a, A_\a)$ such that:

$\hspace{2.2cm}$$(a)$ $q_\a \leq q_*,$

$(*):$$\hspace{1.4cm}$ $(b)$  $A_\a\in D, A_\a \subseteq \k,$

$\hspace{2.2cm}$$(c)$ $q_\a \Vdash ``$ if $i\in A_\a,$ then $ f^1_\a(i) <_{\lusim{J}} \lusim{h}(i)$''.

\begin{claim}
There are $S, \mathcal{U}_*, p_*$ and $A_*$ such that:

$\hspace{1cm}$$(b)$ $S\subseteq \l_1$ is unbounded,

$\hspace{1cm}$$(b)$ If $\a\neq \b$ are from $S$, then $dom(q_\a) \cap dom(q_\b)=\mathcal{U}_*$,

$\hspace{1.cm}$$(c)$ $\a\in S \Rightarrow q_\a \upharpoonright \mathcal{U}_* =p_*,$

$\hspace{1.cm}$$(d)$ $\a\in S \Rightarrow A_\a=A_*.$
\end{claim}
\begin{proof}
By the assumption on $\l_1$ and the $\Delta-$system lemma, we can find $S_1, \mathcal{U}_*$ such that $S_1\subseteq \l_1$ is unbounded, and for all $\a\neq \b$  from $S_1$, $dom(q_\a) \cap dom(q_\b)=\mathcal{U}_*$. Since $|\mathcal{U}_* |<\theta,$ and  $\l_1 > 2^\k = |\{q\in \mathbb{Q}_2: dom(q)=\mathcal{U}_*    \}|,$  we can find an unbounded $S \subseteq S_1$, $p_*\in \mathbb{Q}_2$ and $A_*\in D$ such that the conditions of the claim are satisfied by them.
\end{proof}

For $i<\k$ let $I^*_i=\{f^1_\a(i): \a\in S \}\subseteq I.$ By enlarging $I$ if necessary, we can assume without loss of generality that  $|I| >\l_1$ (in $V$). Define $g\in$ $I^{\k}$ by
\begin{center}
$g(i)=$the $<_I-$least upper bound of $I^*_i.$
\end{center}
\begin{claim}
If $\a<\l_1,$ then $ f^1_\a <_D g.$
\end{claim}
\begin{proof}
Let $\b\in S, \b >\a.$ Then $f_\a <_D f_\b \leq_D g.$
\end{proof}
\begin{claim}
If $\a<\l_2,$ then $g <_D f^2_\a.$
\end{claim}
\begin{proof}
Let $\b\in (\a, \l_2)$ and let $B=\{i<\k: g(i)>_I f^2_\b(i) \}.$ If $B\notin D,$ then $g \leq_D f^2_\b <_D f^2_\a$ and we are done. So suppose that $B\in D.$ For each $i\in B$ there is $t_i\in I^*_i$ such that $ f^2_\b(i) <_I t_i.$ So there is $\sigma_i \in S$ such that $t_i=f^1_{\sigma_i}(i).$ Let $q = \bigcup\{p_{\sigma_i}: i\in B  \}.$ By Claim 5.2, $q$ is a well-defined function and so $q\in \mathbb{Q}_2.$ Further $i\in B \Rightarrow q \leq p_{\sigma_i}$ and by Claim 5.2
\begin{center}
$p_{\sigma_i}\Vdash ``A_*=\{j<\k:  f^1_{\sigma_i}(j) <_{\lusim{J}} \lusim{h}(j)$''$ \}.$
\end{center}
So $A_*\cap B\in D,$and
\begin{center}
$i\in A_*\cap B \Rightarrow q\Vdash ``  f^2_\b(i) \leq_{\lusim{J}}  t_i=f^1_{\sigma_i}(i) <_{\lusim{J}}  \lusim{h}(i)$'',
\end{center}
hence $q\Vdash  f^2_\b <_D \lusim{h}$'',
and we get a contradiction.
\end{proof}
It follows that $g\in V$ is such that for all $\a_1<\l_1$ and $\a_2 <\l_2,$ $f^1_{\a_1} <_D g <_D f^2_{\a_2}$ and we get a contradiction.

{\bf Case 2. The general case:} We now show how to remove the extra assumption $(\forall \a <\l_1) |\a|^\k < \l_1$ from the above proof. Let $\sigma=2^{<\theta}.$ Then $\sigma=\sigma^{<\theta}=\sigma^\k,$ as $\theta$ is regular. Let $\langle (q_\a, A_\a): \a<\l_1  \rangle$ be as in $(*)$.
\begin{claim}
There is $u_* \subseteq \mathcal{U}, |u_*|\leq \sigma$ such that if $u_* \subseteq u\in [\mathcal{U}]^{\leq \sigma}$ and $\a<\l_1,$ then for some $\b\in [\a, \l_1)$ we have $dom(q_\b)\cap u \subseteq u_*.$
\end{claim}
\begin{proof}
Suppose not. We define $(u_\xi, \beta_\xi),$ by induction on $\xi< \theta,$ such that:

$\hspace{1.cm}$$(\a)$ $u_\xi\in [\mathcal{U}]^{\leq \sigma},$

$\hspace{1.cm}$$(\b)$ $\langle u_\zeta: \zeta\leq \xi   \rangle$ is $\subseteq-$increasing and continuous,


$\hspace{1.cm}$$(\gamma)$ $\beta_\xi< \l_1$,

$\hspace{1.cm}$$(\delta)$ $\langle \b_\zeta: \zeta \leq \xi \rangle$ is increasing and continuous,

$\hspace{1.cm}$$(\epsilon)$ If $\xi=\zeta+1$ and $\a\in [\b_\zeta, \l),$ then $dom(p_\a)\cap u_\xi \nsubseteq u_\zeta.$

{\bf Case 1. $\xi=0$:} Let $(u_\xi, \b_\xi)=(\emptyset, 0).$

{\bf Case 2. $\xi$ is a limit ordinal:} Let $u_\xi= \bigcup_{\zeta<\xi}u_\zeta,$ and $\b_\xi=\bigcup_{\zeta<\xi}\b_\zeta.$ Then $|u_\xi|\leq \sigma$ and $\b_\xi <\l_1$ as $|\xi|<\theta=cf(\theta) < \l_1=cf(\l_1).$

{\bf Case 3. $\xi=\zeta+1$ is a successor ordinal:} By our assumption, there are $u, \a$ such that $u_\zeta \subseteq u\in [\mathcal{U}]^{\leq \sigma}, \a <\l_1$ and for all $\b\in [\a, \l_1), dom(q_\b)\cap u \nsubseteq u_\zeta.$ Set $u_\xi=u$ and $\b_\xi=max\{\b_\zeta +1, \a+1  \}.$

Now set $\b= \bigcup_{\xi<\theta}\b_\xi.$ Then $\b<\l_1$ as $\l_1=cf(\l_1)>\theta$ and for all $\xi<\theta, \b_\xi<\l_1.$ Also $\xi<\theta \Rightarrow dom(q_\b)\cap u_{\xi+1} \nsubseteq u_\xi.$ As $\langle u_\xi: \xi <\theta  \rangle$ is $\subseteq-$increasing, it follows that $|dom(q_\b)|\geq \theta,$ which is a contradiction.
\end{proof}
Fix $u_*$ as in Claim 5.5.
\begin{claim}
There are $p_*, S, A_*$ such that:

$\hspace{1.cm}$$(\a)$ $p_*\in \mathbb{Q}_2, p_*\leq q_*,$

$\hspace{1cm}$$(\b)$ $dom(p_*)\setminus dom(q_*) \subseteq u_*,$

$\hspace{1.cm}$$(\gamma)$ $S\subseteq \l_1$ is unbounded in $\l_1,$

$\hspace{1.cm}$$(\delta)$ If $\a\in S,$ then $A_\a=A_*$ and $q_\a \upharpoonright (dom(q_*)\cup u_*)=p_*,$

$\hspace{1.cm}$$(\epsilon)$ If $u \subseteq \mathcal{U}, |u|\leq\sigma$ and $\a<\l_1,$ then there is $\b$ such that $\a <\b\in S$ and $dom(q_\b)$

$\hspace{1.5cm}$is disjoint from $u\setminus dom(p_*).$
\end{claim}
\begin{proof}
Let $\langle (p_\xi, B_\xi): \xi <\xi_*   \rangle$ list $\{(p, B)\in \mathbb{Q}_2\times D: p\leq q_*$ and $dom(p)\setminus dom(q_*)\subseteq u_*  \}.$ As $|u_*|\leq \sigma,$ and members of $\mathbb{Q}_2$ are functions into $\{0,1\}$, clearly $|\xi_*|\leq \sigma^{<\theta}\times 2^\k=\sigma,$ so w.l.o.g $\xi_*\leq \sigma.$ Let
\begin{center}
$S_\xi=\{\a<\l_1: A_\a = B_\xi$ and $q_\a \upharpoonright  (dom(q_*)\cup u_*)=p_\xi    \}.$
\end{center}
So  $\langle S_\xi: \xi <\xi_*    \rangle$ is a partition of $\l_1.$ If for some $\xi, (p_\xi, S_\xi, B_\xi)$ is as required on $(p_*, S, A_*),$ we are done. Suppose otherwise. Clearly, for each $\xi<\xi_*,$ one of the following occurs:
\begin{itemize}
\item $S_\xi$ is bounded in $\l_1.$ Then let $\a_\xi=sup(S_\xi)+1$ and $u_\xi=\emptyset,$
\item $S_\xi$ is unbounded in $\l_1,$ then clause $(\epsilon)$ must fail. Let $u_\xi, \a_\xi$ witness the failure of  $(\epsilon)$.
\end{itemize}
 Let $u= \bigcup_{\xi<\xi_*}u_\xi \cup u_*$ and $\a= sup\{\a_\xi: \xi <\xi_* \}+1.$ Then $u\subseteq \mathcal{U}, |u|\leq \sigma$ and $\a<\l_1.$ By Claim 5.5, there is $\b\in (\a, \l_1)$ such that $dom(q_\b)\cap u \subseteq u_*.$ Pick $\xi<\xi_*$ such that $p_\xi=q_\b \upharpoonright (dom(q_*)\cup u_*)$ and $B_\xi=A_\b.$ So $\a <\b\in S_\xi$ and hence $S_\xi$ is unbounded in $\l_1.$ But then $\a_\xi <\b\in S$ and
\begin{center}
$dom(q_\b) \cap (u_\xi \setminus dom(p_\xi))=\emptyset,$
\end{center}
which is in contradiction with our choice of $u_\xi, \a_\xi.$
The claim follows.
\end{proof}
Fix $p_*, S$ and $A_*$ as above.
For $i<\k$ let $J_i$ be the set of all $t\in I$ such that if $u\subseteq\mathcal{U}, |u|\leq \sigma$ and $\a<\l_1,$ then there is $\b$ such that:

$\hspace{2.3cm}$$(a)$ $t\leq_I f^1_\b(i),$

$(**):$$\hspace{1.5cm}$$(b)$ $\a<\b\in S,$

$\hspace{2.3cm}$$(c)$ $dom(q_\b)$ is disjoint from $u\setminus dom(p_*).$

We also assume w.l.o.g that $I$ is  of cardinality $>\l_1$ and we define $g\in I^{\k}$ by
\begin{center}
$g(i)=$the $<_I-$least upper bound of $J_i.$
\end{center}
\begin{claim}
If $\a<\l_1,$ then $f^1_\a \leq_D g$.
\end{claim}
\begin{proof}
Let $B=\{i<\k: g(i)<_I f^1_\a(i)  \}.$ If $B\notin D,$ we get the desired conclusion, so assume that $B\in D.$ So for every $i\in B, f^1_\a(i)\notin J_i,$ hence there are $u_i\subseteq\mathcal{U}$ of size $\leq \sigma$ and $\a_i<\l_1$ such that there is no $\b$ as requested in $(**),$ for $t=f^1_\a(i).$ Let $u= \bigcup_{i\in B}u_i$ and $\a_*= \bigcup\{\a_i: i\in B \}\cup \a.$ The $u$ is a subset of $\mathcal{U}$ of size $\leq\sigma$, and by Claim 5.6 we can find $\b$ such that $\a_* <\b\in S$ and $dom(q_\b)$ is disjoint from $u\setminus dom(p_*).$ Now $\a\leq \a_* <\b,$ hence $f^1_\a <_D f^1_\b$, hence $C=\{i<\k: f^1_\a(i) <_I f^1_\b(i)\}\in D.$ Hence $B\cap C\in D,$ in particular $B\cap C\neq \emptyset.$ Let $i\in B\cap C.$ Then:

$\hspace{1.cm}$ $(\a)$ $f^1_\a(i) <_I f^1_\b(i),$ as $i\in C,$

$\hspace{1.cm}$ $(\b)$ $f^1_\a(i) \notin J_i,$ as $i\in B,$

$\hspace{1.cm}$ $(\gamma)$ $\a<\nu\in S$ and $dom(q_\nu)\cap u \subseteq dom(p_*) \Rightarrow f^1_\nu(i) <_I f^1_\a(i),$ as $u_i$ witnesses

$\hspace{1.4cm}$ $f^1_\a(i) \notin J_i.$

In particular, as $\b$ satisfies $(\gamma),$  we have $f^1_\b(i) <_I f^1_\a(i).$ But this is in contradiction with $(\a).$
\end{proof}
\begin{claim}
If $\a<\l_2,$ then $g \leq_D f^2_\a.$
\end{claim}
\begin{proof}
Let $B=\{i<\k: f^2_\a(i) <_I g(i)  \}.$ If $B\notin D$ we are done, so assume toward contradiction that $B\in D.$
First note that for $i\in B, f^2_\a(i)\in J_i$. To see this, suppose $u  \subseteq \mathcal{U}$ is of size $\leq \sigma$
and $\alpha < \lambda_1$.
As $i \in B, f^2_\a(i) <_I g(i),$ so by the definition of $g$, we can find $t' \in J_i$ with $f^2_\a(i) \leq_I t'.$ Let $\beta$ witness $t' \in J_i$ with respect to $u$
and $\alpha$.
Then $f^2_\a(i) \leq_I t' \leq_I f^1_\beta(i)$ and both $(b)$ and $(c)$  of  $(**)$ are satisfied for this $\beta.$ Thus $\beta$ witnesses $(**)$
with respect to $t=f^2_\a(i).$ It follows that $f^2_\a(i) \in J_i.$

 Let $p_1\in \mathbb{Q}_2$ and $C\subseteq D$ be such that $p_1\leq p_*$  and $p_1 \Vdash C=\{i<\k: \lusim{h}(i) \leq_{\lusim{J}} f^2_{\a+1}(i) <_{\lusim{J}} f^2_\a(i)\}.$  Clearly $C\in D.$

We define $\b_i$ by induction on $i\in B$ such that:

$\hspace{1.cm}$$(\a)$ $\a< \b_i\in S,$

$\hspace{1.cm}$$(\b)$ $f^2_\a(i) \leq_I f^1_{\b_i}(i),$

$\hspace{1.cm}$$(\gamma)$ $dom(q_{\b_i}) \cap (\bigcup\{dom(q_{\b_j}): j\in B\cap i \}\cup dom(p_1))\subseteq dom(p_*).$

{\bf Case 1. $i=min(B):$} Let $\b_i$ be the least element of $S$ above $\a$ such that $f^2_\a(i) \leq_I f^1_{\b_i}(i).$ Such $\b_i$ exists by definition of $g.$

{\bf Case 2. $i > min(B):$} Suppose $\b_j$ for $j\in B\cap i$ are defined. Let $u=\bigcup\{dom(q_{\b_j}): j\in B\cap i \}\cup dom(p_1).$ Then $u\subseteq \mathcal{U}$ and $|u|\leq \sigma,$ so by definition of $J_i$ and $g$, we can find $\b\in S$ such that $\b > \bigcup_{j\in B\cap i}\b_j$, $f^2_\a(i) \leq_I f^1_{\b}(i),$ and $dom(q_\b)$ is disjoint from $u\setminus dom(p_*).$ Set $\b_i=\b.$

Let $q=\bigcup\{q_{\b_i}: i\in B \}\cup p_1.$ By $(\gamma),$ $q$ is a well-defined function, so
 $q\in \mathbb{Q}_2$. Clearly $q\leq p_1,$ and  $q\leq q_{\b_i},$ for $i\in B.$

As $\b_i\in S,$  $A_{\b_i}=A_*\in D$ (by Claim 5.6$(\delta)$), hence  $A_*\cap B\cap C\in D.$ Let $i\in A_*\cap B \cap C.$ Then:

$\hspace{1.cm}$$(\delta)$ $q\Vdash `` \lusim{h}(i) \leq_{\lusim{J}} f^2_{\a+1}(i) <_{\lusim{J}} f^2_\a(i)$'', as $q\leq p_1$ and $i\in C$,

$\hspace{1.cm}$$(\epsilon)$ $f^2_\a(i) \leq_I f^1_{\b_i}(i),$ by $(\b)$ above,

$\hspace{1.cm}$$(\zeta)$ $q\Vdash `` f^1_{\b_i}(i) \leq_{\lusim{J}} \lusim{h}(i)$'', as $q\leq p_{\b_i}$ and $i\in A_*=A_{\b_i}.$

It follows that
\begin{center}
 $q\Vdash `` f^1_{\b_i}(i)\leq_{\lusim{J}} \lusim{h}(i) <_{\lusim{J}} f^2_\a(i) \leq_{\lusim{J}} f^1_{\b_i}(i)$'',
\end{center}
which is a contradiction.
\end{proof}
\begin{claim}
If $\a_1 <\l_1$ and $\a_2 < \l_2,$ then $ f^1_{\a_1} <_D g <_D f^2_{\a_2}.$
\end{claim}
\begin{proof}
We have $f^1_{\a_1} <_D f^1_{\a_1+1} \leq_D g,$ and $g \leq_D f^2_{\a_2+1} <_D f^2_{\a_2},$ and so we are done.
\end{proof}

Thus $g\in V$ is such that for all $\a_1<\l_1$ and $\a_2 <\l_2,$ $f^1_{\a_1} <_D g <_D f^2_{\a_2}$ and we get a contradiction.
Lemma 5.1 follows.
\end{proof}
\begin{theorem}
Assume $\k < \theta=\theta^{<\theta}< \mu$ and $\mu$ is a supercompact cardinal. Let $\mathbb{Q}=Add(\theta,\mu).$ Then in $V^{\mathbb{Q}},$ we have $2^\theta=\mu, \theta^{<\theta}=\theta >\k$ and for every ultrafilter $D$ on $\k,$ if $(\l_1, \l_2)\in \mathscr{C}(D),$ then $\l_1+\l_2 <\mu.$
\end{theorem}
\begin{proof}
Let $G$ be $\mathbb{Q}-$generic over $V$. Toward contradiction assume that in $V[G],$ there are ultrafilter $D$ on $\k$ and regular cardinals $\l_1, \l_2$ such that $\l_1+\l_2 \geq \mu$ and $(\l_1, \l_2)\in \mathscr{C}(D).$ Assume w.l.o.g that $\l_2\geq \l_1.$ Let $\l=\l_1+\l_2$, and let $I$ be a $\l^+-$saturated dense linear order and let $(\bar{f}^1/D, \bar{f}^2/D)$ witness a pre-cut of $I^\k/D$ of cofinality $(\l_1, \l_2)$, where $\bar{f}^l/D= \langle f^l_\a/D: \a <\l_l \rangle, l=1,2.$
We may assume that the set of elements of $I$ is $|I|,$ so that it belongs to $V.$ It follows that $I^\k \subseteq V,$ and  $f^l_\a \in V$.
\begin{claim}
We can assume that $D\in V.$
\end{claim}
\begin{proof}
Let $\eta < \mu$ be such that $D\in V[G\cap Add(\theta, \eta)].$ Then  $V[G]$ is a generic extension of $V[G\cap Add(\theta, \eta)]$ by $Add(\theta, \mu\setminus\eta),$ and $Add(\theta, \mu\setminus\eta) \simeq Add(\theta, \mu).$ So by replacing $V$ by $V[G\cap Add(\theta, \eta)],$ if necessary, we can assume that $D\in V.$
\end{proof}
By our assumption, we have $\Vdash_{\mathbb{Q}} ``\maltese$'', where

$\hspace{2.cm}$ ``$\lusim{I}$ is a linear order, $D$ is an ultrafilter on $\k$ and $(\lusim{\bar{f}}^1, \lusim{\bar{f}}^2)$

$(\maltese):$$\hspace{1.3cm}$ represents a $(\l_1, \l_2)$-pre-cut in $\lusim{I}^\k /D$ which is a pre-cut

$\hspace{2.cm}$ in $J^\k/D$ for each linear order  $J \supseteq \lusim{I}$'',

and $\lusim{I}, \lusim{\bar{f}^l}, l=1,2$ represent $\mathbb{Q}-$names for $I, \bar{f^l}, l=1,2$ respectively.

Let $j: V \rightarrow M$ be an elementray embedding, witnessing the $\l-$supercompactness of $\mu,$ so that $crit(j)=\mu, M^\l \subseteq M$ and  $\{j(\a): \a<\l_2\}$ is bounded in $j(\l_2).$ Clearly $j$ is the identity on $H(\mu),$ hence $j(\k)=\k, j(\theta)=\theta$ and $j(D)=D.$

Let $\mathbb{Q}_1=\mathbb{Q}$ and $\mathbb{Q}_2=j(\mathbb{Q}).$ Then $M\models$``$\mathbb{Q}_2=Add(\theta, j(\mu))$'',  hence $V \models$``$\mathbb{Q}_2=Add(\theta, j(\mu))$''.
It follows from Lemma 5.1 that  $\Vdash_{\mathbb{Q}_2}$``$\maltese$'', and hence

$(*)$ $\hspace{4.7cm}$ $M\models \Vdash_{\mathbb{Q}_2}$``$\maltese$''.

On the other hand, since
 \begin{center}
 $V\models\Vdash_{\mathbb{Q}_1}``\maltese$'',
 \end{center}
and since $j$ is an elementary embedding, we have

$(**)$ $\hspace{4.7cm}$ $M\models$``$ \Vdash_{\mathbb{Q}_2}``j(\maltese)$'',

where

$\hspace{1.9cm}$ ``$\lusim{j(I)}$ is a linear order, $D$ is an ultrafilter on $\k$ and $(j(\lusim{\bar{f}}^1)/D, j(\lusim{\bar{f}}^2)/D)$

$(j(\maltese)):$$\hspace{.9cm}$ represents a $(j(\l_1), j(\l_2))$-pre-cut in $\lusim{j(I)}^\k /D$ which is a pre-cut in

$\hspace{2.1cm}$ $J^\k/D$ for each linear order  $J \supseteq \lusim{j(I)}$''.

Assume for example $\l_1 \geq \l_2$ and pick $\delta$ such that $\sup\{j(\a): \a < \l_1\} < \delta < j(\l_1).$ Then $\Vdash_{\mathbb{Q}_2}$``$j(\lusim{\bar{f}}^1)_{\delta} \in j(\lusim{I})^\k$'' and for any $\a < \l_1$
and $\gamma < \l_2$,  by 
$(**),$ it is forced by $\MQB_2$  that
\[
j(\lusim{\bar{f}}^1(\a))=j(\lusim{\bar{f}}^1)_{j(\a)} <_{D} j(\lusim{\bar{f}}^1)_{\delta} <_D j(\lusim{\bar{f}}^2)_{j(\gamma)} = j(\lusim{\bar{f}}^2(\gamma)).
\]
By elementarity, we can find $h$ such that for all  $\a < \l_1$
and $\gamma < \l_2$, it is forced that 
\[
\lusim{\bar{f}}^1(\a) <_D h <_D \lusim{\bar{f}}^2(\gamma),
\]
which contradicts $(\maltese).$
\end{proof}

We now give a global version of Theorem 5.10. For this, we need the following generalization of Lemma 5.1.

\begin{lemma}
Suppose that:

$\hspace{1.cm}$$(\a)$ $V_0 \models \k <\theta=cf(\theta),$

$\hspace{1.cm}$$(\b)$ $V_0 \models \l_1, \l_2$ are regular cardinals and $\l_1 + \l_2 > 2^{<\theta},$

$\hspace{1.cm}$$(\gamma)$   $\mathbb{P}\in V_0$ is a $\theta-c.c.$ forcing notion of size $\leq \sigma=2^{<\theta},$

$\hspace{1.cm}$$(\delta)$ $V_1=V_0^{\mathbb{P}},$

$\hspace{1.cm}$$(\epsilon)$ $\mathbb{Q}_l= Add(\theta, \mathcal{U}_l)_{V_0}, l=1,2,$ where $\mathcal{U}_1 \subseteq \mathcal{U}_2$ (hence $\mathbb{Q}_1 \lessdot \mathbb{Q}_2$ in $V_0$),

$\hspace{1.cm}$$(\zeta):$ $\mathbb{R}_l=\mathbb{P}\times \mathbb{Q}_l, l=1,2$ (hence $\mathbb{R}_1 \lessdot \mathbb{R}_2$ in $V_0),$

$\hspace{1.cm}$$(\eta)$  $\lusim{\bar{f}}^l = \langle \lusim{f}^l_\a: \a <\l_l \rangle, l=1,2$ where $\lusim{f}^l_\a$ is an $\mathbb{R}_l-$name for a function from $\k$  to

$\hspace{1.5cm}$$\lusim{I}.$

$\hspace{1.cm}$$(\theta)$ $ \Vdash_{\mathbb{R}_1} ``\maltese$'', where $\maltese$ is as in Lemma 5.1, but the names are $\mathbb{R}_1-$names (and

$\hspace{1.5cm}$hence $\mathbb{R}_2-$names ) here.

Then $ \Vdash_{\mathbb{R}_2} ``\maltese$''.
\end{lemma}

\begin{proof}
We repeat the proof of Lemma 5.1, with some changes. We usually work in $V_1=V_0^{\mathbb{P}},$ but sometimes go back to $V_0$. W.l.o.g. $I\in V_1$ is a complete linear order, whose set of elements is in $V_0$.
Also without loss of generality $\l_1 > 2^{<\theta}$, hence by Theorem 2.16, $\l_2 >2^{<\theta}$. We assume for simplicity that $\mathcal{U}_1=\emptyset,$ so that $\mathbb{Q}_1$ is the trivial forcing notion (see also the proof of Lemma 5.1). Set $\mathcal{U}=\mathcal{U}_2\setminus \mathcal{U}_1=\mathcal{U}_2.$

Since $\mathbb{Q}_2$ is $\theta-$closed and $\theta>\k,$ we have (in $V_1^{\mathbb{Q}_2}$) $I^\k \subseteq V_1,$ and  $D, f^l_\a \in V_1$ (for $l=1,2, \a<\l_l$).

Toward contradiction assume that $\nVdash_{\mathbb{R}_2} `` \maltese$.'' So (working in $V_1$) we can find $q_*\in \mathbb{Q}_2$ and $\mathbb{Q}_2-$names $\lusim{J}$ and $\lusim{h}$ such that:

$\hspace{1.cm}$ $(\a)$ $q_* \in \mathbb{Q}_2,$

$\hspace{1.cm}$ $(\b)$ $q_* \Vdash `` \lusim{J}$ is a linear order such that $\lusim{J} \supseteq I$'',

$\hspace{1.cm}$ $(\gamma)$ $q_*\Vdash `` \lusim{h}\in$ $\lusim{J}^{\k}$'',

$\hspace{1.cm}$ $(\delta)$ $q_*\Vdash ``f^1_{\a_1} <_D \lusim{h} <_D f^2_{\a_2}$ for every $\a_1 < \l_1, \a_2 < \l_2$''.

For each $\a <\l_1$ choose $(q_\a, A_\a)$ such that:

$\hspace{2.2cm}$$(a)$ $q_\a \leq q_*,$

$(*):$$\hspace{1.4cm}$ $(b)$  $A_\a\in D, A_\a \subseteq \k,$

$\hspace{2.2cm}$$(c)$ $q_\a \Vdash ``$ if $i\in A_\a,$ then $ f^1_\a(i) <_{\lusim{J}} \lusim{h}(i)$''.

\begin{claim}
W.l.o.g, $\langle q_\a: \a<\l_1 \rangle\in V_0.$
\end{claim}
\begin{proof}
For each $\a<\l_1,$ we have $q_\a\in \mathbb{Q}_2\in V_0$. It follows that $\langle q_\a: \a<\l_1 \rangle$ is a sequence of elements of $V_0$. Since $\l_1=cf(\l_1) \geq |\mathbb{P}|^+,$ there exists $S\in V_0$ such that $S$ is an unbounded subset of $\l_1$ and such that $\langle q_\a: \a\in S \rangle\in V_0.$ Rearranging this sequence we get $\langle q_\a: \a<\l_1 \rangle\in V_0.$

\end{proof}

\begin{claim} (In $V_1$)
There is $u_* \subseteq \mathcal{U}, |u_*|\leq \sigma$ such that if $u_* \subseteq u\in [\mathcal{U}]^{\leq \sigma}$ and $\a<\l_1,$ then for some $\b\in [\a, \l_1)$ we have $dom(q_\b)\cap u \subseteq u_*.$
\end{claim}
\begin{proof}
Suppose not. Thus for any $u \subseteq [\mathcal{U}]^{\leq \sigma}$ we can find $u \subseteq u' \in [\mathcal{U}]^{\leq \sigma}$ and $\a' < \l_1$
  such that there is no $\b\in [\a', \l_1)$ with $dom(q_\b)\cap u' \subseteq u.$ Note that the ordinal $\a'$ can be taken to be larger than 
   any given ordinal, so in fact, for any  $u \subseteq [\mathcal{U}]^{\leq \sigma}$
   and any $\a < \l_1$, there are  $u \subseteq u' \in [\mathcal{U}]^{\leq \sigma}$ and $\a < \a' < \l_1$
  such that there is no $\b\in [\a', \l_1)$ with $dom(q_\b)\cap u' \subseteq u.$ 
   
   So  in $V_0$, there are $p\in \mathbb{P}$ and $\mathbb{P}-$names $\lusim{F}_1, \lusim{F}_2$ such that $p\Vdash ``$ if $u\in [\mathcal{U}]^{\leq \sigma}$ and $\a<\l_1,$ then $u \subseteq \lusim{F}_1(u,\a)\in [\mathcal{U}]^{\leq \sigma}, \lusim{F}_2(u,\a)\in (\a, \l_1)$ and there is no $\b\in [\lusim{F}_2(u,\a), \l_1)$ such that $dom(q_\b)\cap \lusim{F}_1(u,\a) \subseteq u$''.

As $\mathbb{P}$ has $\theta-c.c.$, $\theta\leq \sigma$ and $\l_1=cf(\l_1)>\theta,$ there are functions $G_1, G_2\in V_0$ such that:

$\hspace{1.cm}$$(\a)$ $p\Vdash ``$ if $u\in [\mathcal{U}]^{\leq \sigma}$ and $\a<\l_1,$ then $\lusim{F}_1(u,\a) \subseteq G_1(u,\a)$ and $\a \leq \lusim{F}_2(u,\a) \leq$

$\hspace{1.5cm}$ $G_2(u,\a)$'',

$\hspace{1.cm}$$(\b)$ $G_1(u,\a)\in ([\mathcal{U}]^{\leq \sigma})^{V_0},$ and $G_2(u,\a)\in [\a,\l_1).$

We define $(u_\xi, \beta_\xi),$ by induction on $\xi< \theta,$ such that:

$\hspace{1.cm}$$(\gamma)$ $u_\xi\in ([\mathcal{U}]^{\leq \sigma})^{V_0},$

$\hspace{1.cm}$$(\delta)$ $\langle u_\zeta: \zeta\leq \xi   \rangle$ is $\subseteq-$increasing and continuous,


$\hspace{1.cm}$$(\epsilon)$ $\beta_\xi< \l_1$,

$\hspace{1.cm}$$(\zeta)$ $\langle \b_\zeta: \zeta \leq \xi \rangle$ is increasing and continuous,

$\hspace{1.cm}$$(\eta)$ If $\xi=\zeta+1$ and $\a\in [\b_\zeta, \l),$ then $dom(q_\a)\cap u_\xi \nsubseteq u_\zeta.$

{\bf Case 1. $\xi=0$:} Let $(u_\xi, \b_\xi)=(\emptyset, 0).$

{\bf Case 2. $\xi$ is a limit ordinal:} Let $u_\xi= \bigcup_{\zeta<\xi}u_\zeta,$ and $\b_\xi=\bigcup_{\zeta<\xi}\b_\zeta.$ Then $|u_\xi|\leq \sigma$ and $\b_\xi <\l_1$ as $|\xi|<\theta=cf(\theta) < \l_1=cf(\l_1).$

{\bf Case 3. $\xi=\zeta+1$ is a successor ordinal:} By the choice of $p$,  $p\Vdash `` u_{\zeta} \subseteq \lusim{F}_1(u_\zeta,\beta_\zeta)\in [\mathcal{U}]^{\leq \sigma}, \lusim{F}_2(u_\zeta,\beta_\zeta)\in (\beta_\zeta, \l_1)$ and there is no $\b\in [\lusim{F}_2(u_\zeta,\beta_\zeta), \l_1)$ such that $dom(q_\b)\cap \lusim{F}_1(u_\zeta,\beta_\zeta) \subseteq u_\zeta$''.

Let $u_\xi=G_1(u_\zeta, \beta_\zeta)$ and $\beta_\xi=G_2(u_\zeta, \b_\zeta)+1.$ Then by $(\a)$, 
\begin{center}
$p\Vdash `` u_\zeta \subseteq u_\xi$ and there is no $\a\in [\b_\xi, \l_1)$ such that $dom(q_\a)\cap u_\xi \subseteq u_\zeta$''.
 \end{center}
 As all parameters in the above formula are from the ground model and the sentence is absolute, it follows that for no $\a\in [\b_\xi, \l_1)$, $dom(q_\a)\cap u_\xi \subseteq u_\zeta.$

Now set $\b= \bigcup_{\xi<\theta}\b_\xi.$ Then $\b<\l_1$ as $\l_1=cf(\l_1)>\theta$ and for all $\xi<\theta, \b_\xi<\l_1.$ Also $\xi<\theta \Rightarrow dom(q_\b)\cap u_{\xi+1} \nsubseteq u_\xi.$ As $\langle u_\xi: \xi <\theta  \rangle$ is $\subseteq-$increasing, it follows that $|dom(q_\b)|\geq \theta,$ which is a contradiction.
\end{proof}
The next claim is the same as Claim 5.6
\begin{claim} (In $V_1$)
There are $p_*, S, A_*$ such that:

$\hspace{1.cm}$$(\a)$ $p_*\in \mathbb{Q}_2, p_*\leq q_*,$

$\hspace{1cm}$$(\b)$ $dom(p_*)\setminus dom(q_*) \subseteq u_*,$

$\hspace{1.cm}$$(\gamma)$ $S\subseteq \l_1$ is unbounded in $\l_1$,

$\hspace{1.cm}$$(\delta)$ If $\a\in S,$ then $A_\a=A_*$ and $q_\a \upharpoonright (dom(q_*)\cup u_*)=p_*,$

$\hspace{1.cm}$$(\epsilon)$ If $u \subseteq \mathcal{U}, |u|\leq\sigma$ and $\a<\l_1,$ then there is $\b$ such that $\a <\b\in S$ and $dom(q_\b)$

$\hspace{1.5cm}$is disjoint from $u\setminus dom(p_*).$
\end{claim}

Fix $p_*, S$ and $A_*$ as above.
For $i<\k$ let $J_i$ be the set of all $t\in I$ such that if $u\subseteq\mathcal{U}, |u|\leq \sigma$ and $\a<\l_1,$ then there is $\b$ such that:

$\hspace{2.3cm}$$(a)$ $t\leq_I f^1_\b(i),$

$(**):$$\hspace{1.5cm}$$(b)$ $\a<\b\in S,$

$\hspace{2.3cm}$$(c)$ $dom(q_\b)$ is disjoint from $u\setminus dom(p_*).$

\begin{claim}
$\langle J_i: i<\k \rangle \in V_1.$
\end{claim}
\begin{proof}
For each $i<\k, J_i\in V_1,$ so as the forcing $\mathbb{Q}_2$ is $\theta-$closed and $\theta>\k, \langle J_i: i<\k \rangle \in V_1.$
\end{proof}

We also assume w.l.o.g. that $I$ is  of cardinality $>\l_1$ and we define $g\in$$I^{\k}$ by
\begin{center}
$g(i)=$the $<_I-$least upper bound of $J_i.$
\end{center}
As $g$ is defined using the sequence $\langle J_i: i<\k \rangle$, it follows from Claim 5.16 that:
\begin{claim}
$g\in V_1.$
\end{claim}
The next claim can be proved as in Claim 5.7.
\begin{claim}
If $\a<\l_1,$ then $f^1_\a \leq_D g$.
\end{claim}

\begin{claim}
If $\a<\l_2,$ then $g \leq_D f^2_\a.$
\end{claim}
\begin{proof}
The proof is the same as the proof of Claim 5.8. We just need to choose $\b_i$ to be minimal so that the condition $q$ defined there is in $V_0$ and hence $q\in \mathbb{Q}_2.$
\end{proof}
It follows that
\begin{claim}
If $\a_1 <\l_1$ and $\a_2 < \l_2,$ then $ f^1_{\a_1} <_D g <_D f^2_{\a_2}.$
\end{claim}

Thus $g\in V_1$ is such that for all $\a_1<\l_1$ and $\a_2 <\l_2,$ $f^1_{\a_1} <_D g <_D f^2_{\a_2}$ and we get a contradiction.
Lemma 5.12 follows.
\end{proof}

\begin{theorem}
If in $V$, there is a class of supercompact cardinals, then for some class forcing $\mathbb{P},$ in $V^\mathbb{P}$ we have: for any infinite cardinal $\k$, and  any ultrafilter $D$ on $\k,$ if $(\l_1, \l_2)\in \mathscr{C}(D),$ then $\l_1 +\l_2 < 2^{2^\k}.$
\end{theorem}

\begin{proof}
By a preliminary forcing (see \cite{laver}), we can assume that the following hold in $V$, for some proper class $C$ of cardinals:

$\hspace{1.cm}$$(\a)$ Each $\k\in C$ is a supercompact cardinal,

$\hspace{1.cm}$$(\b)$ No limit point of $C$ is an inaccessible cardinal,

$\hspace{1.cm}$$(\gamma)$ $\k\in C \Rightarrow \k$ is Laver indestructible,

$\hspace{1.cm}$$(\delta)$ $\k\in C \Rightarrow 2^\k=\k^+.$

We choose cardinals $\k_i, i\in Ord,$ by induction on $i$ as follows:

{\bf Case 1. $i=0$:} Let $\k_0=\aleph_0,$

{\bf Case 2. $i$ is a limit ordinal:} Let $\k_i=\bigcup_{j<i}\k_j,$

{\bf Case 3. $i=j+1$ is a successor ordinal, $\k_j$ is $\aleph_0$ or a supercompact cardinal:} Let $\k_i=\k_j^+,$

{\bf Case 4.  $i=j+1$ is a successor ordinal and case 3 does not hold:} Let $\k_i$ be the minimal element of $C$ above $\k_j.$

Note that by  $(\b),$ $\k_i$ is a singular cardinal iff $i$ is a limit ordinal.

Let $\mathbb{Q}_i$ be $Add(\k_i, \k_{i+1}),$ if $\k_i$ is regular, and the trivial forcing otherwise.
 Let $\mathbb{Q}$ be the Easton support product of $\langle \mathbb{Q}_i: i\in Ord \rangle,$ and let $\mathbb{Q}_{<j}$ and  $\mathbb{Q}_{>j}$ be defined similarly for $\langle \mathbb{Q}_i: i<j \rangle$ and $\langle \mathbb{Q}_i: i>j \rangle$ respectively. By standard forcing arguments we have:
\begin{claim} Let $G$ be $\mathbb{Q}-$generic over $V$, and for each ordinal $i$ set $G_{<i}=G\cap \mathbb{Q}_{<i}$ and $G_{>i}=G\cap \mathbb{Q}_{>i}.$ Then:

$(a)$ $V$ and $V[G]$ have the same cardinals,

$(b)$ If $\l < \k_i,$ then $P(\l)^{V[G]}=P(\l)^{V[G_{<i}]},$

$(c)$ If $\k_i$ is regular in $V$, then $|\mathbb{Q}_{<i}| \leq \k_i$, and in $V[G], \k_i^{<\k_i}=\k_i$ and $2^{\k_i}=\k_{i+1}.$
\end{claim}
In $V[G]$, let $\k$ be an infinite cardinal and let $D$ be an ultrafilter on $\k.$ Let $i$ be the least ordinal such that $\k <\k_i.$ Then $i=j+1$ is a successor ordinal, and we have $2^{2^\k}=2^{2^{\k_j}}=2^{\k_i}=\k_{i+1},$ so it suffices to prove the following:
\begin{claim}
(In $V[G])$  $(\l_1, \l_2)\in \mathscr{C}(D) \Rightarrow \l_1+\l_2 <\k_{i+1}.$
\end{claim}
\begin{proof}
Write $\mathbb{Q}=\mathbb{Q}_{<i+1} \times \mathbb{Q}_{>i}.$ The forcing $\mathbb{Q}_{>i}$ is $\k_{i+1}-$directed closed, so by $(\gamma), \k_{i+1}$ remains supercompact in $V[G_{>i}].$ Let $V_0=V[G_{>i}]$. Note that:

$\hspace{1.cm}$$(d)$ $\mathbb{Q}_i=\mathbb{Q}_{V_0}$ and $\mathbb{Q}_{<i}=(\mathbb{Q}_{<i})_{V_0},$

$\hspace{1.cm}$$(e)$ $V_0 \models \mathbb{Q}_{<i}$ is $\k_i-c.c.$ of size $\leq \k_i.$

The rest of the argument is essentially the same as the proof of Theorem 5.10, using Lemma 5.12 instead of Lemma 5.1. So toward contradiction assume that in $V[G],$ $(\l_1, \l_2)\in \mathscr{C}(D)$ is such that $\l_1+\l_2 \geq \k_{i+1}.$
Let $I$ be a $(\l_1+\l_2)^+-$saturated dense linear order and let $(\bar{f}^1/D, \bar{f}^2)/D$ witness a pre-cut of $I^\k/D$ of cofinality $(\l_1, \l_2)$, where $\bar{f}^l/D= \langle f^l_\a/D: \a <\l_l \rangle, l=1,2.$ We may assume that the set of elements of $I$ is in $V.$ From now on we work in $V_0$. Set $\mathbb{P}=\mathbb{Q}_{<i}$ and $\mu=\k_{i+1}.$

We also suppose that $\Vdash^{V_0}_{\mathbb{P}\times\mathbb{Q}_i} ``\maltese$'', where

$\hspace{2.cm}$ $\lusim{I}$ is a linear order, $\lusim{D}$ is an ultrafilter on $\k$ and $(\lusim{\bar{f}}^1/\lusim{D}, \lusim{\bar{f}}^2/\lusim{D})$

$(\maltese):$$\hspace{1.3cm}$ represents a $(\l_1, \l_2)-$pre-cut in $\lusim{I}^\k /\lusim{D}$ which is also a pre-cut

$\hspace{2.cm}$ in $J^\k/\lusim{D}$ for each linear order  $J \supseteq \lusim{I}$'',

and $\lusim{I}, \lusim{\bar{f}^l}\in V_0, l=1,2$ represent $\mathbb{P}\times \mathbb{Q}_i-$names for $I, \bar{f^l}, l=1,2$ respectively (over $V_0$).

 Let $\l=\l_1+\l_2$  so that $\l\geq\mu$ is regular. Let $j: V_0  \rightarrow M_0$ be an elementary embedding witnessing the $\l-$supercompactness of $\mu$; so  that $crit(j)=\mu, j(\mu)>\l$ and $M_0^\l \subseteq M_0.$

Since
 \begin{center}
 $V_0 \models \Vdash_{\mathbb{P}\times\mathbb{Q}_i}``\maltese$'',
 \end{center}
and since $j$ is an elementary embedding and $j(\mathbb{P})=\mathbb{P},$ we have

$(*)$ $\hspace{4.2cm}$ $M_0 \models \Vdash_{\mathbb{P}\times j(\mathbb{Q}_i)}``j(\maltese)$'',

where

$\hspace{2.1cm}$ $\lusim{j(I)}$ is a linear order, $\lusim{D}$ is an ultrafilter on $\k$ and $(j(\lusim{\bar{f}}^1)/\lusim{D}, j(\lusim{\bar{f}}^2)/\lusim{D})$

$(j(\maltese)):$$\hspace{.9cm}$ represents a $(j(\l_1), j(\l_2))-$pre-cut in $\lusim{j(I)}^\k /\lusim{D}$ which is also a pre-cut

$\hspace{2.1cm}$ in $J^\k/\lusim{D}$ for each linear order  $J \supseteq \lusim{j(I)}$'',

On the other hand, it follows from Lemma 5.12 that  $\Vdash_{\mathbb{P}\times j(\mathbb{Q}_i)}``\maltese$'', and hence

$(**)$ $\hspace{4.2cm}$ $M_0\models \Vdash_{\mathbb{P}\times j(\mathbb{Q}_i)}``\maltese$''.

From $(*)$ and $(**)$ we can get the required contradiction as in the proof of Theorem 5.10. The claim follows.
\end{proof}
Theorem 5.21 follows.
\end{proof}

\subsection*{Acknowledgements}
The authors thank the referee of the paper for many helpful comments and corrections.

\begin{center}

\end{center}

School of Mathematics, Institute for Research in Fundamental Sciences (IPM), P.O. Box:
19395-5746, Tehran-Iran.

E-mail address: golshani.m@gmail.com

\begin{center}

\end{center}

Einstein Institute of Mathematics,  The Hebrew University
of Jerusalem, Jerusalem, 91904, Israel, and Department of Mathematics, Rutgers University, New
Brunswick, NJ 08854, USA.

E-mail address: shelah@math.huji.ac.il


\begin{thebibliography}{99}

\bibitem{erdos} Erd\"{o}s, Paul; Hajnal, Andr\'{a}s; M\'{a}t\'{e}, Attila; Rado, Richard, Combinatorial set theory: partition relations for cardinals, Studies in Logic and the Foundations of Mathematics 106, Amsterdam: North-Holland Publishing Co. 1984.

\bibitem{garti-shelah1} Garti, Shimon and Shelah, Saharon, On Depth and Depth$^+$ of Boolean algebras. Algebra Universalis 58 (2008), no. 2, 243-–248.

\bibitem{garti-shelah2} Garti, Shimon and Shelah, Saharon, Depth of Boolean algebras. Notre Dame J. Form. Log. 52 (2011), no. 3, 307–-314.

\bibitem{garti-shelah3} Garti, Shimon and Shelah, Saharon, $(\k, \theta)-$weak normality. J. Math. Soc. Japan 64 (2012), no. 2, 549-–559.

\bibitem{garti-shelah4} Garti, Shimon and Shelah, Saharon, Depth$^+$ and Length$^+$ of Boolean algebras, submitted.

\bibitem{jech}  Jech, Thomas, Set theory. The third millennium edition, revised and expanded. Springer Monographs in Mathematics. Springer-Verlag, Berlin, 2003. xiv+769 pp.

\bibitem{laver} Laver, Richard, Making the supercompactness of $\k$ indestructible under $\k-$directed closed forcing. Israel J. Math. 29 (1978), no. 4, 385–-388.

\bibitem{magidor} Magidor, Menachem, On the singular cardinals problem. I. Israel J. Math. 28 (1977), no. 1--2, 1-–31.

\bibitem{malliaris-shelah1}  Malliaris, Maryanthe and  Shelah, Saharon, Model-theoretic properties of ultrafilters
built by independent families of functions, Journal of Symbolic Logic, accepted.

\bibitem{malliaris-shelah2}  Malliaris, Maryanthe and  Shelah, Saharon, Constructing regular ultrafilters from a model-theoretic point of view, Trans. Amer. Math. Soc., accepted.

\bibitem {malliaris-shelah3} Malliaris, Maryanthe and  Shelah, Saharon, Cofinality spectrum theorems in model theory, set theory, and general topology, preprint.

\bibitem{shelah3}  Shelah, Saharon, Classification theory and the number of nonisomorphic models. Second edition. Studies in Logic and the Foundations of Mathematics, 92. North-Holland Publishing Co., Amsterdam, 1990. xxxiv+705 pp.

\bibitem{shelah-g}  Shelah, Saharon, Cardinal arithmetic. Oxford Logic Guides, 29. Oxford Science Publications. The Clarendon Press, Oxford University Press, New York, 1994. xxxii+481 pp.

\bibitem{shelah4} Shelah, Saharon, The depth of ultraproducts of Boolean Algebras, Algebra Universalis 54 (2005), 91–-96.

\bibitem{solovay}  Solovay, Robert M. Strongly compact cardinals and the GCH. Proceedings of the Tarski Symposium (Proc. Sympos. Pure Math., Vol. XXV, Univ. California, Berkeley, Calif., 1971), pp. 365–-372. Amer. Math. Soc., Providence, R.I., 1974.

\end{thebibliography}
\end{document}